\documentclass[a4paper,11pt]{article}
\usepackage[utf8x]{inputenc}
\usepackage{amsmath, relsize} 
\usepackage{todonotes} 
\usepackage{amsthm} 
\usepackage{amsfonts}
\usepackage{bbm}
\usepackage{amssymb}
\usepackage{array}
\usepackage{url}
\usepackage{xcolor}
\usepackage{hyperref}
\usepackage{graphicx} 
\usepackage{caption}
\usepackage{subcaption}
\usepackage{enumitem}
\usepackage{varwidth}

\usepackage{lipsum}
\usepackage{authblk}
\usepackage[top=2.5cm, bottom=2cm, left=1.5cm, right=1.5cm]{geometry}
\usepackage{fancyhdr}
\let\emptyset\varnothing

\hypersetup{
    colorlinks,
  linkcolor={red!50!black},   
    citecolor={blue!50!black},
    urlcolor={blue!80!black}
}

\newtheorem{theorem}{Theorem}[section]
\newtheorem{proposition}[theorem]{Proposition}
\newtheorem{corollary}[theorem]{Corollary}
\newtheorem{lemma}[theorem]{Lemma}
\newtheorem{remark}[theorem]{Remark}

\newtheorem{definition}{Definition}

\newtheorem*{theorem*}{Theorem}
\newtheorem*{proposition*}{Proposition}
\newtheorem*{corollary*}{Corollary}
\newtheorem*{lemma*}{Lemma}
\newtheorem*{remark*}{Remark}
\newtheorem*{claim*}{Claim}
\newtheorem*{definition*}{Definition}

\numberwithin{equation}{section}

\newcommand{\caA}{{\mathcal A}}

\newcommand{\caC}{{\mathcal C}}

\newcommand{\caF}{{\mathcal F}}

\newcommand{\caH}{{\mathcal H}}

\newcommand{\caJ}{{\mathcal J}}

\newcommand{\caL}{{\mathcal L}}

\newcommand{\caN}{{\mathcal N}}

\newcommand{\caP}{{\mathcal P}}
\newcommand{\caQ}{{\mathcal Q}}
\newcommand{\caR}{{\mathcal R}}
\newcommand{\caS}{{\mathcal S}}

\def\bbone{{\mathchoice {\rm 1\mskip-4mu l} {\rm 1\mskip-4mu l} {\rm 1\mskip-4.5mu l} {\rm 1\mskip-5mu l}}}

\newcommand{\bbE}{{\mathbb E}}

\newcommand{\bbN}{{\mathbb N}}

\newcommand{\bbP}{{\mathbb P}}

\newcommand{\bbR}{{\mathbb R}}

\newcommand{\bbZ}{{\mathbb Z}}

\newcommand{\e}[1]{\,{\rm e}^{#1}\,}
\newcommand{\Or}{{\rm Or}}
\newcommand{\Orb}{{\rm Orb}}

\newcommand{\fc}[1]{{\bar {#1}}}
\renewcommand{\sc}[1]{{\hat {#1}}}
\newcommand{\eps}{\varepsilon}
\newcommand{\dd}{{\rm d}}

\begin{document}

\title{Scaling limit of a self-avoiding walk \\ 
interacting with spatial random permutations }
\author[]{Volker Betz}
\author[]{Lorenzo Taggi}
\affil[]{Technische Universit\"at Darmstadt,  Germany}
\date{} 

\maketitle

\abstract{
We consider nearest neighbour spatial random permutations on $\bbZ^d$. 
In this case, the energy of the system is proportional the sum of all cycle lengths, and the system can be interpreted 
as an ensemble of edge-weighted, mutually self-avoiding loops. The 
constant of proportionality, $\alpha$, is the order parameter of the model. 
Our first result is that in a parameter regime of edge weights 
where it is known that a single self-avoiding loop is weakly space filling, 
long cycles of spatial random permutations are still exponentially unlikely. 
For our second result, we embed a self-avoiding walk into a background of spatial 
random permutations, and condition it to cover a macroscopic distance. 
For large values of $\alpha$ (where long cycles are very unlikely) we show that 
this walk collapses to a straight line in the scaling limit, and give bounds on 
the fluctuations that are almost sufficient for diffusive scaling. For proving our 
results, we develop the concepts of spatial 
strong Markov property and iterative sampling 
for spatial random permutations, which may be of independent interest. Among other 
things, we use them to show exponential decay of correlations for large values of 
$\alpha$ in great generality. 
}

\vspace{6mm}

\rhead{ \large{ \textit{1 \hspace{0.7cm} INTRODUCTION}}}
\section{Introduction}

Self-avoiding random walks are by now a classical topic of 
modern probability theory, although many questions still remain 
to be answered; we refer to the classic book \cite{MadrasSlade}
and the more recent survey \cite{Slade2011}. A variant of 
self-avoiding walks are self-avoiding polygons (see e.g.\ 
\cite{Hammersley1961}), where the last
step of the self-avoiding walk has to come back to the point 
of origin. 

Spatial random permutations, on the other hand, 
are a relatively recent concept. They were originally introduced due to 
their relevance for the theory of Bose-Einstein condensation 
\cite{Betz4, BU_Ph, Uel06}, 
but are of independent mathematical interest. The purpose of the present paper 
is to view spatial random permutations as systems of mutually self-avoiding 
polygons and compare the behaviour of a selected cycle of a spatial random 
permutation to the one of a self-avoiding walk or polygon. Put differently, 
a selected cycle of a spatial random permutation can be viewed as a self-avoiding 
polygon embedded into a background of other self-avoiding polygons, and we are 
interested in the effect that this embedding has. For this purpose, we restrict 
to nearest neighbour spatial random permutations, as they are most closely related
to self-avoiding walks. For a finite subset $\Lambda_n$ of $\bbZ^d$, the model 
is defined by the probability space $\caS_{\Lambda}$ consisting of all bijective 
maps $\pi$ on $\Lambda$ with the property that $|\pi(x)-x| \in \{0,1\}$. The 
probability measure (with order parameter $\alpha$) is given by assigning 
the element $\pi \in \caS_\Lambda$ the energy 
$
\caH(\pi) := \sum_{x \in \Lambda} | \pi(x) - x |,
$
and the probability 
\[
\bbP_\Lambda(\{\pi\}) = \frac{1}{Z(\Lambda)} \e{-\alpha \caH(\pi)},
\]
where $Z(\Lambda)$ is the partition function (normalising constant).
 
The most important questions in spatial random permutations concern the length 
of their cycles, in particular the existence of macroscopic cycles. It is known \cite{Betz4}
that the probability that the origin is in a cycle of length larger than $k$ 
decays exponentially with $k$, uniformly in the volume $\Lambda$. 
It is expected that for dimensions $d=3$ and higher, there exists a critical value 
$\alpha_c$ of the order parameter so that for $\alpha < \alpha_c$, the cycle
containing the origin (or any selected point) is of macroscopic length with 
positive probability. In $d=2$, on the other hand, only a Kosterlitz-Thouless 
phase transition is expected, meaning that the probability of the cycle 
being larger than $k$ decays exponentially if $\alpha > \alpha_c$ but 
algebraically otherwise.  
While there is good numerical evidence for the existence of long cycles in 
$d \geq 3$ \cite{Gandolfo, Grosskinsky} and the Kosterlitz-Thouless transition \cite{Betz} in $d=2$, actually proving any positive 
statement about the existence of long cycles is the great unsolved problem of the 
theory of spatial random permutations. The only case where such a statement is 
known is for an annealed version of the model (with a slightly different energy)
\cite{Betz3,Betz5}.
The argument there relies on explicit calculations using Fourier transforms; all 
attempts to get away from this exactly solvable situation have so far failed. 
In \cite{BU_Ph} a non-rigorous argument is made that the study 
of models of non-spatial permutations with cycle weights may be useful,
and such models have received some attention recently 
\cite{Betz2, BoZei, ErcUel}. 

In the context of self-avoiding walks, the concept of a phase transition from 
short to long loops is present in the following results. Consider a 
single step-weighted, i.e.\ fix a sequence of growing subsets 
$\Lambda_n$  (e.g. cubes) of $\bbZ^d$, for each $n$ 
two points $a$ and $z$ at opposite ends of $\Lambda_n$, and 
consider the set of all self-avoiding walks starting in $a$ and 
ending in $z$. Let $\alpha \in \bbR$, and assign to each such 
self-avoiding walk $X$ the weight $\exp(-\alpha |X|)$, where $|X|$
is the number of steps that $X$ takes. Write $\mu$ for the 
connective constant of the $d$-dimensional cubic lattice. 
When $\alpha > \log \mu$, it is known that the shape of 
$X$ converges to a straight line as $n \to \infty$, when scaled
by $1/n$. Actually, when scaled by $1/n$ 
in the direction of $a-z$ and by $n^{-1/2}$ in the directions
perpendicular to this vector, $X$ converges to a Brownian Bridge. 
These results are implicit in the works \cite{Chayes,Ioffe} 
and have been worked out by Y.\ Kovchegov in his thesis 
\cite{Kovchegov}. 

For $\alpha < \log \mu$, on the other hand, the results are 
entirely different. As Dominil-Copin, Kozma and Yadin have recently shown 
\cite{Copin}, in this case the rescaled self-avoiding walk becomes 
\textit{weakly space filling}, 
meaning that it will only leave holes of logarithmic size in the graph
\cite{Copin}. Their results also hold for the self-avoiding loop, i.e.\ in 
the case where $a$ and $z$ are chosen to be the same site. Note that for 
$\alpha > \log \mu$, the self-avoiding \textit{loop} will converge to a point
in the scaling limit. 

It is therefore a natural question what happens when the self-avoiding loop 
is embedded into a background of other self-avoiding loops, i.e.\ in the 
case of spatial random permutations. The proof of the existence of weakly
space filling cycles would be particularly interesting, as this would imply
at least imply that the expected length of the cycle is infinite, and thus a 
a phase transition. Unfortunately, this is not what we can show. 
Instead, we give a somewhat negative result: we show that 
{\em if} there is a regime of space filling cycles, 
it must start at lower $\alpha$ than for the case of the self-avoiding polygon. 
More precisely, in the case where $G$ is a subgraph of a vertex transitive graph, 
and when $\mu$ is the {\em cyclic} connective constant of that graph, then
we identify in Theorem \ref{theo1} an $\alpha_0 < \log \mu$ so that for all $\alpha > 
\alpha_0$, and uniformly in the size 
of $G$, the length of a cycle through a given point has exponential tails. Thus
in the interval $(\alpha_0, \log \mu)$ the single self-avoiding loop is weakly 
space filling
while the self-avoiding loop embedded into an ensemble of other such loops is very 
short. 

For our second result, we restrict to the case where $\Lambda_n = [0,n] \times
[-n/2,n/2)^{d-1}$ and impose periodic boundary conditions on all 
except the first coordinate. We embed into the spatial random permutation 
a self-avoiding path starting in 
$0$ and conditioned to end in a point $z$ on the opposite side of $\Lambda$, see Figure \ref{Fig:pi}.  
\begin{figure}
    \centering
\includegraphics[scale=0.44]{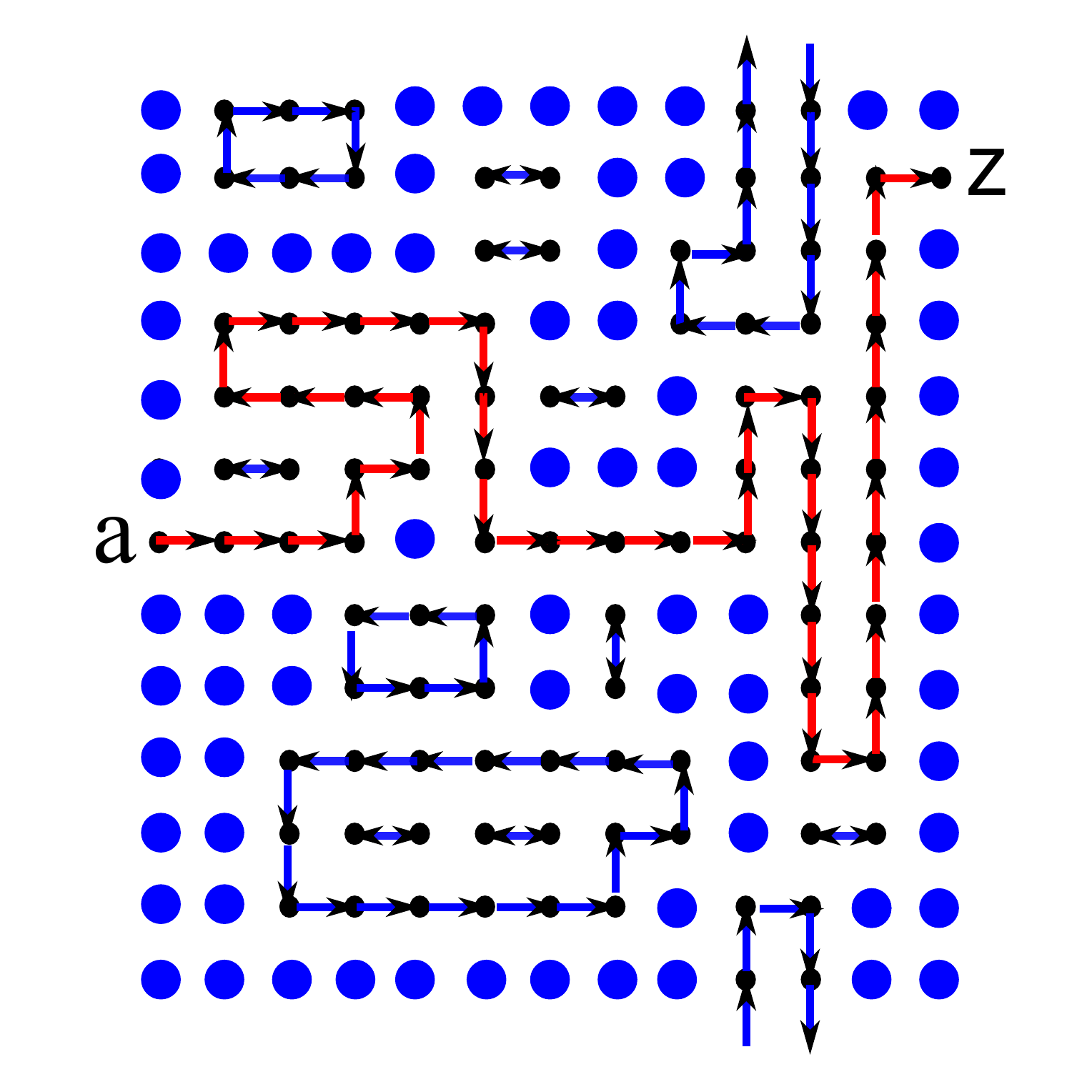}
\caption{
Representation of a bijection with a forced
open cycle between $a$ and $z$,
when $\Lambda \subset \mathbb{Z}^2$ is a box with cylinder
boundary conditions.
If $\pi(x)=x$, then a circle is drawn at $x$, while if $\pi(x)$ is a neighbour of $x$, 
then an arrow is directed from $x$ to $\pi(x)$.}
\label{Fig:pi}
\end{figure}
In this situation, we (almost) recover the results of \cite{Kovchegov}, 
i.e.\ we show that for sufficiently large $\alpha$, 
the self-avoiding walk starting in $0$ collapses to a straight line in the 
scaling limit, as $n \to \infty$. Unlike in the case of the 
single self-avoiding walk, we do not have a good  quantitative estimate on the 
threshold above which this behaviour holds, 
and we cannot quite control the fluctuations 
well enough to prove the convergence to a Brownian bridge. The reason is that
the background of cycles introduces additional correlations that are very 
hard to control. This can be explained well by investigating  the strategy 
of proof for the case of the self-avoiding 
walk, and discussing where it fails in our case. 

The basic idea in the case of the self-avoiding walk (which we indeed adapt and extend) is to introduce regeneration 
points. The walk is forced to connect $0$ to the other 
side of the box; let us assume that it is the other side with respect to the 
first coordinate. 
A regeneration point is a point where the 'past' of the walk 
is entirely to the left (has smaller first coordinate) of that point, while the 
'future' is entirely to the right. It is easy to see that if there are no 
regeneration points in an area of a given (horizontal) width $w$, then 
the self-avoiding path has to have at least $3 w$ steps taking place in this
area. Since every step that the walk takes introduces 
an additional factor of $\e{-\alpha}$ to its weight, and since $\alpha$ is large, 
it is possible to show that there must be many regeneration points, more precisely 
that the probability to not find any regeneration point in a vertical 
strip of width 
$K$ steps decays exponentially in $K$, 
uniformly in $n$. On the other hand, regeneration points provide fresh starts 
for the walk (hence the name): when conditioned to start at a regeneration point, 
the resulting model is again a self-avoiding walk. The regeneration points 
themselves thus form random walk with iid steps, 
and by the exponential bound discussed above we 
have excellent control on the step size of that process. 
All scaling results now follow from standard limit theorems for random walks. 

This argument breaks down at several places when going to spatial random 
permutations. First of all, the energy of the system is an extensive quantity 
for all values of $\alpha$, i.e. it grows proportional to the volume. 
Since the energy  
of the embedded random walk only grows proportional to its length, in the case 
where it collapses to a straight line its energy will be a subdominant 
compared to  energy provided by the environment. Therefore, 
it will be much harder to argue that a broad vertical strip with no regeneration
points is unlikely just because the walk would need to take many steps. 
Secondly, even if we do have a regeneration point, just considering the 
area to the right of it will not decouple the past from the future. 
The reason is that with probability one (in the limit $n \to \infty$) a short 
cycle from the background will cross the vertical hyperplane containing the 
regeneration point, and introduce correlations. 

We solve both problems by developing a method for estimating the decay of 
correlations for spatial random permutations in the case of large 
$\alpha$. Since we expect this method to be of independent interest, and since 
it does not complicate matters much, we develop it for general graphs. The method
is based on a strong version of the spatial Markov property, and on iterative 
sampling, and is strong enough to provide exponential decay of correlations 
in a rather general context. We then extend the concept of regeneration points
and introduce (random) regeneration surfaces, that serve the same purpose as 
the (deterministic) hyperplanes that separate regions in the self-avoiding 
walk case. See Figure \ref{Fig:pi} b) 
for an illustration. While our estimates are strong enough to give us the 
correct scaling limit, they are (just) not strong enough to get down to diffusive 
scaling. The main reason is that  
we can only show that consecutive regeneration points (and surfaces) have a 
distance of order $\log n$ with high probability, not a finite one as 
in the case of the self-avoiding walk.

There are other models where the scaling limit of lower dimensional structures 
interacting with an environment is studied. These include the 
Ising model \cite{Campanino1}, Bernoulli Percolation \cite{Campanino2},
Self-Avoiding Walks \cite{Chayes, Ioffe} and Random Cluster models 
\cite{CIV08}. All of these approaches rely on some variant of the 
Ornstein-Zernike theory \cite{Ornstein}, and they all need correlation
inequalities of some kind. For example, in the Ising model, the 
Edwards-Sokal coupling \cite{Edwards} can be used. No correlation 
inequalities are known in spatial random permutations, and  indeed we 
expect that finding such inequalities would have a significant impact 
on the subject area. Our method shares some features of the 
Ornstein-Zernike approaches above, but does not require correlation 
inequalities in order to apply. Instead, we obtain decay of 
correlations by iterative sampling, although this method is not 
(yet?) quite strong enough in order for us to obtain diffusive 
scaling.

The significance of the model of spatial random permutations 
with a forced cycle goes far beyond the situation that we 
describe here. In \cite{Uel06} it is shown that (in a suitable 
variant of spatial random permutations) the ratio of 
the partition functions of a system with a forced cycle and one 
without can be used to detect Bose-Einstein condensation: if 
this ratio stays positive uniformly in the volume and the spatial
separation of the two endpoints of the forced cycle, this is 
equivalent to the presence of off-diagonal long range 
order \cite{PenOns}, which itself is equivalent to Bose-Einstein 
condensation. Some progress has been made 
in understanding related models, e.g.\ the 
Heisenberg model and its connection to ensembles of mutually self-
avoiding walks \cite{Toth}, or
the random stirring models introduced by Harris \cite{Harris} and 
further analysed in
\cite{Berestycki, Schramm, KoteckiMilosUeltschiet}. The big question 
about the existence of long loops, however, remains open. 

On the probability side, our model is closely related to the  
connection is to the  loop $O(n)$ model,
which has been introduced in \cite{Domany}.
In this model a loop configuration  $\omega$ is an undirected spanning subgraph 
of a graph $G$ such that every vertex of this subgraph has degree zero or two.
The weight of a loop configuration is proportional to $e^{- \alpha \,  o(\omega)} \, n^{L(\omega)}$, 
where $o(\omega)$ is the number of edges in $\omega$, $L(\omega)$ 
is the number of loops and $n$ is a positive real.
The case $n=0$ corresponds formally to self-avoiding walk if one forces a path in the system 
in addition to the loops. 
The case $n=1$ corresponds to the low-temperature representation of the Ising model.
If viewed as an ensemble of cycles, spatial random permutations are 
intimately related to the loop 
$O(n)$ model with $n=2$, since each cycle of the permutation 
admits two possible orientations. 
The two models would be equivalent if in spatial random
permutations cycles of length two were forbidden.
On the hexagonal lattice, the loop $O(n)$ model has been conjectured to undergo
a Kosterlitz-Thouless phase transition at the critical threshold
$\log \big(  \sqrt{2 + \sqrt{2 - n}  } \big)$ when $n \leq 2$ \cite{Nienhuis}.
This is compatible with our general finding that on every vertex-transitive graph 
the critical threshold of spatial random permutations, which corresponds
more or less to the $n=2$ case, is strictly less than the critical
threshold for the self-avoiding walk, corresponding to the $n=0$ case.
Furthermore, it has been conjectured that only short cycles 
are observed at all values of $\alpha$ when $n>2$.
This has been rigorously proved only for $n$ large enough in the 
 article of Duminil-Copin, Peled, Samotij, and Spinka in \cite{Copin2},
who also provide details on the structure of the typical configurations and
provide evidence for the occurrence of a phase transition.
Exploring the properties of the model at low values of $n$ is of great interest
and entirely open. 
Most of the proofs that we present in this paper,
can be reproduced for the loop $O(n)$ model for all values of $n$
and $\alpha$ large enough, without 
the restriction of considering the hexagonal lattice
(see Remark \ref{remark:loopO(n)} below for further comments on this).

Our paper is organized as follows. We give precise definitions and 
state our results in Section \ref{sect:definitions}. 
In Section \ref{sect:cyclelength} we 
prove our result on non-existence of long cycles, and provide
various estimates on the partition functions over different domains
which we will need later.  
In Section \ref{sect:cyclegrowth} we 
discuss the spatial Markov property and iterative
sampling, and derive our results about 
exponential decay of correlations.
In Section \ref{sect:proofoftheo2}, we give the proof of our main 
result,  Theorem \ref{theo2}, using the results of the previous 
sections.

\rhead{ \large{ \textit{2 \hspace{0.7cm} DEFINITIONS AND RESULTS}}}
\section{Definitions and main results}
\label{sect:definitions}

We consider a finite simple graph $G = (V,E)$. 
A permutation on $G$ is a bijective map $\pi: V \to V$ so that for all $x \in V$, 
either $\pi(x)=x$ or $\{x,\pi(x)\} \in E$. 
We write $\mathcal{S}_V$ for the set of all permutations on $G$, omitting the dependence on
the edge set in the notation. Also, when $U \subset V$, we  write $\mathcal{S}_U$ for 
the set of permutations on the subgraph $(U, E_U)$ generated by $U$, i.e.\ where 
$E_U := \{(x,y) \in E: x \in U, y \in U\})$. 

For a given $\pi \in \mathcal{S}_V$, we define its \textit{energy} by 
\begin{equation}
\label{eq:energy}
\mathcal{H}_{V}(\pi) = \sum\limits_{x \in V}  \mathbbm{1}\{   \pi(x) \neq x \},
\end{equation}
where the indicator above is $1$ if $\pi(x) \neq x$ 
and $0$ if $\pi(x) = x$, and its probability as
\begin{equation}
\label{eq:measureG}
\mathbb{P}_V (\pi) = \frac{\exp \big( - \alpha \, \mathcal{H}_{V}(\pi)  \big) }{Z(V)}.
\end{equation}
Above, $\alpha \in \mathbb{R}$ controls the preference of random permutations to have fixed 
points, and $Z(V)$ is called \textit{partition function}.

The expression \eqref{eq:energy} is less general than it could be. By replacing the 
indicator function in that formula with  
edge weights $d:E \to E, \{x,y\} \mapsto d(x,y)$, we can generalize the definition of 
random permutations on graphs sufficiently so that classical cases including the quadratic 
jump penalization (see e.g. \cite{Betz,Betz4}) are covered. While some of our results below 
(most notably the iterative sampling procedure) can be adapted to hold for this general 
situation, our main results compare the behavior of random permutations to the behavior 
of self-avoiding paths. Therefore, we prefer to stick with the narrower definition 
\eqref{eq:energy}.

The most interesting objects in spatial random permutations are their cycles. For 
$\pi \in \mathcal{S}_V$ and $z \in V$, the \textit{cycle of $\pi$ containing $z$} is the 
directed graph on  
$$
\{ \, \, \pi^i(z) \in V \, \, \, : \, \, \, i \in \mathbb{N} \, \, \}
$$
with edge set 
$$
\{   \, \, (\pi^i(z), \pi^{i+1}(z))  \in E \, \, : i \in \mathbb{N} \, \, \}.
$$
We will denote it by $\gamma_z(\pi)$. We will regularly abuse the notation and use the 
symbol $\gamma_z(\pi)$ also to denote the vertex set of the cycle, viewed as a subset of $V$.
We denote the total number of edges of $\gamma_z(\pi)$ by $\| \gamma_z \|$
and we refer to it as \textit{length of $\gamma_z$}. In the special case where 
$\pi(z)=z$, $\gamma_z(\pi)$ has vertex set $\{z\}$ and empty edge set, therefore length $0$.
It is known \cite{Betz4} that there exists some $\alpha_0> 0$ so that for all $\alpha > 
\alpha_0$, long cycles are exponentially
unlikely. The first result of the present paper is to sharpen this statement by providing 
some information about the value of $\alpha_0$. To state it, let us recall the definitions
of connective constant and cyclic connective constant. 

A \textit{self-avoiding path} in a graph $G$ is a finite directed 
subgraph $(U,E')$ of $G$ such that 
there is an enumeration  $(x^1, x^2, \ldots x^n)$ of $U$ with the property that 
\[
E' = \{( x^i, x^{i+1} ) : 1 \leq i \leq n-1\}.
\]
A \textit{cycle} in $G$ is a finite directed subgraph $(U,E^{\prime})$ of $G$ such that
there is an enumeration  $(x^1, x^2, \ldots x^n)$ of $U$ with the property that 
\[
E' = \{( x^i, x^{i+1} ) : 1 \leq i \leq n-1\} \cup \{  \, ( x^n, x^1) \, \}.
\]
Note that by considering the graph $(U,E')$ as directed,
we give $\gamma$ an orientation. For an infinite, 
vertex transitive 
graph ($\mathbb Z^d$ with the nearest neighbor edge structure 
being the most important example), we single out a vertex $0 \in V$ and call it the origin. 
We write $SAW_n$ (respectively $SAP_n$) for the set of all self-avoiding paths (respectively  cycles)
starting from $0$,  with $n$ edges.  Thus, we have that $SAW_0 = SAP_0 = \{0\}$.
Then the limits
\begin{equation}
\label{eq:cycleconnective}
\mu_{G} = \limsup\limits_{n \rightarrow \infty} \sqrt[n]{ |SAP_n|},
\end{equation}
and  
\begin{equation}
\label{eq:connective}
\mu^{\prime}_{G} = \lim_{n \rightarrow \infty} \sqrt[n]{ |SAW_n|}
\end{equation}
exist. For (\ref{eq:connective}) this follows from a sub-additivity argument \cite{Hammersley1954}, while for \eqref{eq:cycleconnective} it follows from the fact that
$|SAP_n| \leq |SAW_n|$. The latter also immediately shows 
$\mu_{G} \leq \mu^{\prime}_{G}$ for all vertex-transitive graphs $G$. 
Hammersley \cite{Hammersley1961} proved the remarkable fact that 
\[
\mu^{\prime}_{\mathbb{Z}^d} = \mu_{\mathbb{Z}^d}.
\]
We are now ready to state our first main result.
\begin{theorem}  
\label{theo1}
Let $G$ be any infinite vertex-transitive graph of bounded degree.
Let $\alpha_0$ be the unique solution of the equation 
\begin{equation}\label{eq: def of alpha c}
\alpha + \tfrac12 \log (1 + \exp(- 2 \alpha)) = \log \mu_G.
\end{equation}
Then for all $\alpha > \alpha_0 $  there exist constants 
$C_0(\alpha), c_0(\alpha) > 0$ such that for any finite subgraph of $G$ generated by 
$U \subset V$, for all $z \in U$, and for all $\ell \in \mathbb N$,
 \begin{equation}
 \label{eq:theo1}
 \mathbb{P}_{U}( \| \gamma_z \| >\ell  ) \leq  \hypertarget{C_0}{C_0(\alpha)}  \exp \big (  - \hypertarget{c_0}{c_0(\alpha)} \, \,   \ell  \big ) .
\end{equation}
Moreover, we can choose $C_0(\alpha)$ and $c_0(\alpha)$ so that $\lim_{\alpha \to \infty} c_0(\alpha) / \alpha  = 1$ and 
$\limsup_{\alpha \to \infty} C_0(\alpha) < \infty$. 
\end{theorem}

It is interesting to compare the above result with the findings of 
Duminil-Copin, Kozma and Yadin
\cite{Copin}, who study self-avoiding walks on $\mathbb Z^d$
with an energy proportional to their length. To connect to the present paper, it is best to 
view them as random permutations conditional on not having any cycles except the one through
the origin. In \cite{Copin} it is shown that when $\alpha < \log \mu'_{\mathbb{Z}^d}$, the resulting 
self-avoiding cycle is {\em weakly space filling}, in particular its expected length is 
infinite. Theorem \ref{theo1} shows that without the conditioning on having just one cycle,
the situation is drastically different: since $\mu'_{\mathbb Z^d} = 
\mu_{\mathbb Z^d}$ and since the solution $\alpha_0$ of \eqref{eq: def of alpha c} is 
strictly smaller than $\log \mu_{\mathbb{Z}^d}$, in the interval $(\alpha_0, \log \mu_{\mathbb Z^d})$
the lonely self-avoiding cycle has infinite expected length 
(as the relevant subgraph becomes large), while the length of the 
cycle through the origin in ordinary spatial random permutations has exponential tails.

For our second main result, we restrict our attention to cylindrical subgraphs of 
$\bbZ^d$. For $n\in\bbN$, let 
\[
\Lambda_n := [0,n] \times (-n/2,n/2]^{d-1} \cap \bbZ^d.
\]
Elements of $\Lambda_n$ will be written in the form $x = (\fc{x}, \sc{x})$ with 
$\fc{x} \in \bbZ$ and $\sc x \in \bbZ^{d-1}$. We impose cylindrical boundary 
conditions on $[-n/2,n/2) \cap \bbZ^{d-1}$ (but not on $[0,n] \cap \bbZ$), 
and edges are then between nearest neighbours in $\Lambda_n$. We will denote 
the resulting graph with the same symbol $\Lambda_n$ if no confusion can arise. 
For a subset $A \subset \Lambda_n$, we will also write $A$ for the subgraph of 
$\Lambda_n$ that retains all of the edges where both endpoints lie in $A$. For 
$a,z \in A$, we define $\caS_A^{a \to z}$ as the set of maps $\pi: A \to A$ 
with the properties \\[1mm]
(i): $\pi$ is a bijection from $A \setminus \{z\}$ to $A \setminus \{a\}$.\\[1mm]
(ii): $\pi(z) = z$\\[1mm]
(iii): $|\pi(x)-x| \leq 1$ for all $x \in A$. \\[1mm] 
It is easy to see that $\caS_A^{a \to z} = \emptyset$ if the vertex $a$ is disconnected 
from $z$ in the graph $A$, and that $z = \lim_{n\to\infty} \pi^n(a)$ otherwise. 
For given $\pi \in \caS_A^{a \to z}$, we will always write 
$\gamma(\pi) = \Orb_\pi(\{a\})$. We have 
$\pi(A) = A \setminus \{a\}$, and 
$\gamma(\pi)$ is the trace of a self-avoiding walk starting 
at $a$ and ending in $z$ that is embedded in $\pi$. 
Thus the probability measure $\bbP_A^{a \to z}$ defined through 
\begin{equation} \label{open cycle measure}
\bbP_A^{a \to z} ( \{ \pi \}) = \frac{1}{Z^{a \to z}(A)} \exp \left( 
- \alpha \sum_{x \in A} |\pi(x)-x| \right), \qquad (\pi \in \caS_A^{a \to z}),
\end{equation} 
describes a step-weighted self-avoiding walk interacting with a background of 
spatial random permutations. We also define 
\[
\ell_j := \{ x \in \Lambda_n: \fc x = j \}, \qquad \text{and} \quad 
\caS_A^{a \to \ell_j} := \bigcup_{z \in \ell_n} \caS_{A \cup \{z\}}^{a \to z}.
\]
The probability measure on $\caS_A^{a \to \ell_n}$ will be 
\eqref{open cycle measure}, except that the normalisation is now given by  
$Z^{a \to \ell_m}(A) = \sum_{ z: \fc z = j} Z^{a \to z}(A)$. We can now state 
the main result of this paper. 

\begin{theorem} \label{theo2}
There exists $\alpha_0 > 0$, $D < \infty$ 
and $N \in \bbN$ so that for all $M > 0$, 
\[
\sup_{n \geq N, \alpha > \alpha_0} \bbP_{\Lambda_n}^{0 \to \ell_n} 
\Big( \max \{ | \hat y |: y \in \gamma \} > M \sqrt{n \log n} \Big) < 
D / M.
\]	
\end{theorem}

Thus for large $\alpha$, the self-avoiding walk embedded into 
$\pi$ converges to a straight horizontal line, and the vertical aberration 
can be proved to be just a bit larger than $\sqrt{n}$. Indeed, we expect that 
the true vertical aberration is exactly of the order $\sqrt{n}$ and that 
$\gamma$ converges to a Brownian motion under diffusive scaling. This is 
known in the case of a self-avoiding walk without a background of spatial 
random permutations \cite{Kovchegov}. In our situation, the strong correlations
prevent us from getting a presumably sharp upper bound on the fluctuations, and 
indeed also prevent a useful lower bound. We will comment on the places where 
we lose the necessary accuracy for diffusive behaviour at the end of the proof 
of Theorem \ref{theo2}.

\rhead{ \large{ \textit{3 \hspace{0.7cm} CYCLE LENGTH AND PARTITION FUNCTION}}}
\section{Cycle length and partition function}
\label{sect:cyclelength}
In this section we prove Theorem \ref{theo1} and provide
further estimates comparing partition functions for different subsets
of $V$. 
Our Theorem \ref{theo1}  treats only the case of vertex transitive graphs since 
we want quantitative estimates involving the connective constant; but it is 
not difficult to modify our proof so that one can treat general graphs, 
including those with edge weights. In the latter case, some modifications 
will be necessary, as the graph distance is no longer a good quantity 
to measure the distance between sets. We do not pursue this any further in 
the present paper. 

Recall that for a self-avoiding path  or cycle $\gamma$, $\|\gamma  \|$ denotes the number of its edges and that
$|\gamma|$ denotes the number of its vertices.
Our first comparison is 
\begin{proposition}
\label{prop:estimategamma}
For any finite simple graph $G=(V,E)$,
for any self-avoiding path or cycle $\gamma \subset G$,
we have that
 \begin{equation}
 \label{eq:lowerboundZ}
\frac{Z(V \setminus \gamma)}{Z(V)} \leq { \left( \, \, \frac{1}{1 + e^{- 2 \alpha}} \, \, \right) }^{ \frac{ \|\gamma \|}{2}}.
\end{equation}
\end{proposition}
\begin{proof}
Let $\gamma \subset V$ be a
self-avoiding path
or a cycle.
Let us denote by 
$\{ x^0, x^1, \ldots, x^{|\gamma|-1} \}$ 
the sequence of sites of $\gamma$, ordered 
such that $(x^i, x^{i+1})$ is an edge of $\gamma$.
We let $M$ be the largest even number such that $M \leq |\gamma |$.
We claim that 
\begin{equation}
\label{eq:estimationsU}
{Z} (\gamma) \geq (  1 + e^{- 2 \alpha} )^{ \frac{M}{2}},
\end{equation}
where we recall our abuse of notation: $\gamma$ denotes the  
sites occupied by the cycle $\gamma$, and $Z(\gamma)$ is the partition function of the
subgraph generated by those sites. To see \eqref{eq:estimationsU},
we partition $\gamma$
into pairs $(x^0, x^1)$, $(x^2, x^3)$, $\ldots$, $(x^{M-2}, x^{M-1})$.
Let $\hat \gamma$ be the graph obtained from $\gamma$ by keeping
only edges connecting vertices of the pair $(x^i, x^{i+1})$,
for even $i \in [0, M-1]$,  and removing all the other edges.
Clearly,
$ {Z} (\gamma) \geq {Z} (\hat \gamma)$,
since $\hat \gamma$ is a subgraph of $\gamma$.
Since the graph $\hat \gamma$ is composed of $\frac{M}{2}$ disjoint
subgraphs  containing two vertices connected by an edge each
and since the contribution of each of these subgraphs
is $(1 + e^{- 2 \alpha} )$, (\ref{eq:estimationsU}) follows.

Now note that, by Proposition \ref{prop:Markovspatial} (i), 
\begin{equation}
\label{eq:inequalitypartition1}
Z (V) 
\geq 
Z (A) \,  Z(V \setminus A).
\end{equation}
Note also that if $\gamma$ is a cycle, then 
$|\gamma | = \|  \gamma \|$ is even and we can set $M = \| \gamma \|$,
while if  $\gamma$ is a self-avoiding path, then $M \geq |\gamma| - 1 = \| \gamma \| $. 
Thus, we have that  if $\gamma$ is a self-avoiding path or a cycle, then
$$
\frac{Z(V \setminus \gamma) }{Z(V)} \leq \frac{1}{Z(\gamma)} \leq  
{ \left( \, \, \frac{1}{1 + e^{- 2 \alpha}} \, \, \right) }^{ \frac{\| \gamma \|}{2}}.
$$
\end{proof}

The estimate of Proposition \ref{prop:estimategamma} is 
logarithmically sharp if the cycle $\gamma$ is ``stretched out'' in the sense that 
the subgraph generated by the sites of $\gamma$ contains no 
further edges beyond those of $\gamma$. If $\gamma$ is ``curly'', 
meaning that many points in the relevant subgraph 
are connected by more than two edges, one could use these edges 
in order to get better lower bounds on $Z(A)$ and thus better upper 
bounds on the ratio $Z(V \setminus \gamma) / Z(V)$. The associated 
combinatorics do not look easy even in the case of 
$V \subset \bbZ^2$, though.

We are now ready to prove our first main theorem. 
\begin{proof}[\textbf{Proof of Theorem \ref{theo1}}] 
Let $G$ be an infinite vertex-transitive simple graph
of bounded degree, let $U$ be a finite subset of $V$. 
For $\pi \in \mathcal{S}_{U}$, $x \in U$, and a cycle
$\tilde \gamma \subset U$,  we have 
\begin{equation}
\label{eq:probabilitycycle}
\mathbb{P}_{U}( {\gamma}_{x} = \tilde{\gamma}  ) =    
e^{ - \alpha   \| \tilde{\gamma} \|   } \,   \frac{ Z(U \setminus 
\tilde{\gamma})}{Z (U)},
\end{equation}
Let $\mu_{G}$ be the cyclic connective constant of $G$.
The definition of cyclic connective constant (\ref{eq:cycleconnective})
implies that for every $\delta>0$ there exists $C_\delta>0$ such that
for any $n \in \mathbb{N}$,
$$
|SAP_n| \leq C_\delta \, ( \mu_{G} + \delta)^n.
$$
By vertex transitivity, the same bound holds for $SAP_n(x)$, the set of self-avoiding cycles of length $n$ starting 
in $x$. By Proposition \ref{prop:estimategamma}, we then have that
 for all $\ell \in \mathbb{N}^+$,
\begin{align}
\mathbb{P}_{U}( \| \gamma_{x} \|  \geq \ell  )
 &  = \sum\limits_{n=\ell }^{\infty} \sum\limits_{\tilde{\gamma} \in SAP_n(x) \cap U } \exp \big ( -\alpha \, n \, \big )
 \frac{ Z(U \setminus \tilde{\gamma} )}{Z (U)}  \\
 \label{eq:twosums2}
 & \leq \, C_\delta \,  \, \sum\limits_{n = \ell }^{\infty}  \exp \bigg ( \, -n \, \Big( \alpha + \frac{1}{2} 
 \log \big( \, 1 + e^{- 2 \alpha} \, \big) - \log \big(  \mu_{G} + \delta \big) \, \, \Big ) \, \bigg)
\end{align}
Recall that $\alpha_0$ is defined as the unique solution of (\ref{eq: def of alpha c})
and that therefore it satisfies  $ 0 \, < \, \alpha_0 \, < \, \log(\mu_{G})$.
For each $\alpha > \alpha_0$, we can find $\delta(\alpha, G)>0$ small enough such that
\begin{equation}
\label{eq:defc2}
c_0(\alpha) := \alpha + \frac{1}{2} \log \big(  1 + e^{-2 \alpha} \big) - \log \big( \, \mu_{G} + \delta(\alpha, G) \, \big) > 0.
\end{equation}
Thus 
\[
\mathbb{P}_{U}( \| \gamma_{x} \|  \geq \ell  ) \leq 
 \frac{C_{\delta(\alpha,G)}}{1 - \e{-c_0(\alpha)}} \e{- \ell \, c_0(\alpha)}.
\]
It is not difficult to see that for all large enough $\alpha$, we can choose 
$\delta(\alpha, G) = 1$, and that then 
$\lim_{\alpha \to \infty} c_0(\alpha)/\alpha = 1.$ This concludes the proof 
of the theorem. 
\end{proof}

\begin{remark}
\label{remark:loopO(n)}
Proposition \ref{prop:estimategamma}
is the only point of this paper
where the existence of cycles of length $2$ is necessary.
For the loop $O(2)$ model \cite{Copin2}, corresponding to random permutations
without cycles of length $2$, the exponential bound
(\ref{eq:theo1}) of Theorem \ref{theo1}  would still hold true, but only for $\alpha > \log \mu$.
Exponential decay of cycle length for the loop $O(n)$ model has  been proved in
the paper \cite{Copin2} for any value of $\alpha$ when $n$ is large.
\end{remark}

The next proposition uses similar ideas as in the previous proof in order to 
give a complementary bound
to the one of Proposition \ref{prop:estimategamma}. 
\begin{proposition}
\label{prop:lowerboundpart}
Let $G$ be an infinite vertex-transitive simple graph,
let $\alpha_0$ be the unique solution of equation (\ref{eq: def of alpha c}).
For any $\alpha > \alpha_0$ there exists
$c_1(\alpha)>0$ such that for any finite $U \subset V$, and for all $A \subset  U$,
\begin{equation}
\frac{Z(U \setminus A)}{Z(U)}  \geq \hypertarget{c_1}{c_1(\alpha)}^{|A|}.
\end{equation}
Moreover $c_1(\alpha)$ can be chosen such that $\lim\limits_{\alpha \to \infty} 
c_1(\alpha) = 1$.
\end{proposition}
\begin{proof}
Fix $x \in U$. 
Using the notation as in the proof of Proposition \ref{prop:estimategamma},
we get
\[
Z(U) =  
\sum\limits_{n=0}^{\infty}  \, \, \, 
\sum\limits_{\substack{ \tilde{\gamma} \subset SAP_n(x) \, \cap \,  U } }  \, \, \,
\sum\limits_{\substack{ \pi \in \mathcal{S}_{U} : \\  \gamma_x(\pi) = \tilde{\gamma }  }} \, \, \, 
e^{- \alpha \mathcal{H}_{U}(\pi)}  =
 Z(U \setminus \{x \})    \, \, 
\sum\limits_{n=0}^{\infty}  \sum\limits_{\substack{ \tilde{\gamma} \subset SAP_n(x) \, \cap \,  U}  }    e^{- \alpha \| \tilde \gamma \|}
 \frac{Z(U \setminus  \tilde{\gamma}  ) }{{Z(U \setminus \{x\})}}
\]
With $\tilde \gamma \in SAP_n(x) \cap U$ for $n \geq 2$, 
we apply Proposition \ref{prop:estimategamma} 
to the self-avoiding path $\tilde \gamma \setminus \{x\}$ 
(of length $n-2$), and the graph $U \setminus \{x\}$, 
giving 
\[
\frac{Z(U \setminus \tilde \gamma)}{Z(U \setminus \{x\})}\leq 
 (1 + \e{-2\alpha})^{-(n-2)/2} \quad \text{ for all }  
\tilde \gamma \in SAP_n(x),  \, \, n \geq 2.
\] 
Thus,
\[ 
\frac{Z(U)}{Z(U\setminus\{x\})} \leq 1 + (1 + \e{-2\alpha}) 
\sum_{n=2}^\infty |SAP_n| \e{- n \alpha} (1 + \e{-2 \alpha})^{-n/2}
\leq 1 + \frac{(1 + \e{-2\alpha}) C_{\delta(\alpha,G)}}{1 - \e{-c_0(\alpha)}}
\e{- 2 c_0(\alpha)}.
\]
Above, the constants are the same as in Theorem \ref{theo1}. 
By our knowledge of $C_{\delta(\alpha,G)}$ and $c_0(\alpha)$, clearly 
the right hand side of the above equation converges to one as 
$\alpha \to \infty$. Thus we have shown the claim for $A = \{x\}$, 
with 
\begin{equation}
\label{eq:c1}	
c_1(\alpha) = \Big( 1 + \frac{(1 + \e{-2\alpha}) C_{\delta(\alpha,G)}}{1 - \e{-c_0(\alpha)}}
\e{- 2 c_0(\alpha)} \Big)^{-1}.
\end{equation}
For general $A$, the claim follows from the telescopic product 
\[
\frac{Z \left(U \setminus A \right)}{Z(U)} = 
\frac{Z(U \setminus A)}{Z\big(   U \setminus ( A \setminus \{x^1\}) \big) }  
 \frac{Z(\big(  U \setminus ( A \setminus \{x^1\}) \big)  }
{Z \big(  U \setminus ( A \setminus \{x^1, x^2\}) \big)  }  \ldots  
\frac{Z( U \setminus \{x^{|A|} \} )}{Z(U)}
 \geq 
 \hyperlink{c_1}{c_1(\alpha)}^{|A|}.
\]
\end{proof}


\rhead{ \large{ \textit{4 \hspace{0.7cm} ITERATIVE SAMPLING}}}
\section{Iterative sampling and the spatial Markov property}
\label{sect:cyclegrowth}

Spatial random permutations have a fundamental spatial Markov property, which 
we will first discuss in a relatively simple form. Then we introduce a sort of strong Markov property, 
and finally we present 
iterative sampling, which is a basic new technique that enables many of our proofs. 

Let $G = (V,E)$ be a finite graph. For $B \subset V$, we define the sigma algebra 
$$
\mathcal{F}_B = \sigma\left(  \{ \pi \,  :  \, \pi(x)=y, \, \pi^{-1}(x)=z \} \,  :  \, x \in B, \, y,z \in V   \right).
$$
Put differently, the value $f(\pi)$ of an $\mathcal F_B$-measurable function is determined by the set of values  $\pi(x)$ and $\pi^{-1}(x)$, $x \in B$. For $\pi \in S_V$, we let 
\[
{\rm Inv}(\pi) := \{A \subset V: \pi(A) = A\}
\]
be the family of $\pi$-invariant sets. The  Markov property, statement $(iii)$ in 
the theorem below, says that conditional on $A \in {\rm Inv}$, an $\caF_A$ measurable 
random variable and an $\caF_{A^c}$-measurable random variable are independent. In order 
to state it correctly, let us recall that $\bbP_A$ denotes the probability measure 
\eqref{eq:measureG} on the subgraph generated by $A$, and write $\bbE_A$ for the corresponding expectation. For $\pi \in S_A$ we write 
$\pi \oplus {\rm id}$ for the permutation in $S_V$ that is obtained by setting 
$\pi(x)=x$ for all $x \notin A$. 

\begin{proposition}[\textbf{Spatial Markov property}]
\label{prop:Markovspatial}
Let $A \subset V$, $A \neq \emptyset$. Then, \\[3mm]
(i): $\bbP_V(A \in {\rm Inv}) = Z(A) Z(A^c) / Z(V)$.\\[2mm]
(ii): For $\caF_A$-measurable $f$, we have $\bbE_V(f | A \in {\rm Inv}) = \bbE_A(f(. \oplus {\rm 
id}))$.\\[2mm]
(iii): For $\caF_A$-measurable $f$ and $\caF_{A^c}$-measurable $g$, we have 
\[
\bbE_V(f \, g | A \in {\rm Inv}) = \bbE_V( f | A \in {\rm Inv}) \, \bbE_V( g | A \in {\rm 
Inv}) = \bbE_A(f(. \oplus {\rm id})) \,  \bbE_{A^c}(g({. \oplus  \rm id} )).
\]
(iv): For $\caF_A$-measurable $f$ and $\mathcal{B} \in \caF_{A^c}$, we have 
\begin{equation}
\label{eq:spatialMarkov2}
\bbE_V(f | A \in {\rm Inv}, \mathcal{B}) = \bbE_A(f).
\end{equation}
\end{proposition}

\begin{proof}
	By the additivity of $\caH_V$, we have for $\caF_A$-measurable $f$ and 
	$\caF_{A^c}$-measurable $g$ that 
\begin{equation}
	\label{eq: main markov formula}
	\sum_{\pi \in S_V: \pi(A) = A} f(\pi) g(\pi)\e{- \alpha \mathcal{H}_V(\pi)} = \sum_{\pi \in S_A} f(\pi \oplus {\rm id}) \e{- \alpha \mathcal{H}_A(\pi)} \sum_{\tilde \pi \in S_{A^c}} g({\tilde \pi \oplus \rm id})
	\e{-\alpha \mathcal{H}_{A^c}(\tilde \pi)}.
\end{equation}
Inserting $f=1$ and $g=1$ into \eqref{eq: main markov formula} and dividing by $Z(V)$
gives (i). Setting $g=1$ in \eqref{eq: main markov formula} and dividing by 
$Z(V)$ and using (i) gives 
$\bbE_V(f \mathbbm{1}\{A \in {\rm Inv}\}) = \bbE_A(f) \bbP(A \in {\rm Inv})$, 
and thus (ii). 
For (iii), note that dividing \eqref{eq: main markov formula} by
$Z(A) Z(A^{c})$ gives $\bbE_V(f \, g | A \in {\rm Inv})$ on the left hand side by 
using (i), while on the right hand side it gives 
$\bbE_A(f(. \oplus {\rm id})) \,  \bbE_{A^c}(g(  . \oplus   {\rm id} ))$. The remaining 
equality follows from (ii). For (iv), it suffices to use the equality 
$ \bbE_V(f | A \in {\rm Inv}, \mathcal{B}) = \bbE_V(f \mathbbm{1}\{\mathcal{B}\} | A \in {\rm Inv}) / 
\bbP_V(B | A \in {\rm Inv})$, the fact that $A \in {\rm Inv}$ is equivalent to 
$A^c \in {\rm Inv}$, and (iii).
\end{proof}

Following ideas from the theory of Markov chains, we can also formulate a strong 
Markov property for random permutations. A set-valued random variable 
$Q: S_V \to \caP(V)$ will be called an {\em admissible random invariant set} if 
$Q(\pi) \in {\rm Inv}(\pi)$ for all $\pi \in S_V$, and if the event $\{Q=A\}$ is 
$\caF_A$-measurable for all $A \subset V$. Moreover, we define the sigma algebra
\[
\caF_Q = \{ R \in \caP(S_V): R \cap \{Q = A\} \in \caF_A 
\text{ for all } A \subset V\}.
\]
We then have 
\begin{proposition}[Strong Markov property]
\label{prop:strong Markov}
	Let $Q$ be an admissible random invariant set.
	For $B \subset V$ and $\caF_B$-measurable $f: S_V \to \bbR$ we have 
	\[
	\bbE_V(f | \caF_Q) \, \mathbbm{1}\{Q \cap B = \emptyset\} = \bbE_V(f \, \mathbbm{1}\{Q \cap B = \emptyset\} | \caF_Q) = \bbE_{Q^c} \big( f(. \oplus {\rm id} ) \big)
	\mathbbm{1} \{Q \cap B = \emptyset \}
	\]
\end{proposition}

\begin{proof}
	Since $Q$ is admissible, it is $\caF_Q$-measurable, and so the right hand side 
	above is also $\caF_Q$-measurable. Moreover, for a $\caF_Q$-measurable random 
	variable $g$, we have 
	\[
	\begin{split}
		\bbE_V \Big( \, \bbE_{Q^c}( f(. \oplus {\rm id}) \big)\mathbbm{1}\{Q \cap B = \emptyset\} g \Big)& = 
		\sum_{A: A \cap B = \emptyset} \bbE_V\Big( \bbE_{A^c}(f(. \oplus {\rm id}))  
		\mathbbm{1}\{Q =A \} g \Big)= \\
		& = \sum_{A: A \cap B = \emptyset} \bbE_{A^c}\Big(f(. \oplus {\rm id})\Big) \, \, 
		\bbE_V\Big(\mathbbm{1}\{Q=A\} g \Big| A \in {\rm Inv}\Big) \, \, 
		\bbP(A \in {\rm Inv}) \\ 
		& = \sum_{A: A \cap B = \emptyset} \bbE_{A^c}\Big(f(. \oplus {\rm id})\Big) 
		\,\, 
		\bbE_A\Big(\mathbbm{1}\{Q =A\} g(  . \oplus {\rm id} )\Big) \bbP(A \in {\rm Inv}) \\
		& = \sum_{A: A \cap B = \emptyset} \bbE_V\Big(f g \mathbbm{1}\{Q=A\}\Big)
		= \bbE_V\Big(f g \mathbbm{1}\{Q \cap B = \emptyset\}\Big).
		\end{split}
	\]
	Above, the second equality is true since $Q=A$ implies $A \in {\rm Inv}$ and
	the third equality holds by Proposition \ref{prop:Markovspatial} (ii) and the fact that 
	$\mathbbm{1}\{Q =A\} g$ is $\caF_A$-measurable for each $\caF_Q$-measurable $g$. The fourth 
	equality is due to Proposition \ref{prop:Markovspatial} (iii), which is applicable since $f$ is 
	$\caF_B$-measurable and thus in particular $\caF_{A^c}$-measurable as $A \cap B = \emptyset$. 
	The claim is thus shown. 
\end{proof}

We demonstrate the usefulness of the strong Markov property by applying it to the 
question of decay of dependence on the boundary conditions of spatial permutations. 
Let $G = (V,E)$ be a (large) graph, think of $V = \bbZ^d \cap [-N,N]^d$ with the 
nearest neighbour edge structure. Let $B \subset U \subset V$, 
where we think of $U$ as being large and differing from $V$ only 
``near the boundary'', and of $B$ as being ``near the center'' of $V$. 
We are interested in the difference between $\bbE_{V}(f)$
and $\bbE_{U}(f)$ in cases where the graph distance between $U^c$ 
and $B$ becomes large and $f$ is $\mathcal{F}_B$-measurable.
In other words, we are interested in how much the precise 
shape of the graph at the boundary influences the expectation of 
local random variables. This is a classical question in any theory related to Gibbs 
measures. 
Our answer to this question is the crude, but useful estimate
that is provided in the next proposition.

Let $\tilde G$ be the disjoint union of two copies 
of $G$. Let $\tilde V = V_1 \cup V_2$ be its vertex set, where $V_1$ and $V_2$ are 
disjoint copies of $V$. 
Let $\phi:\tilde V \to \tilde V$ be the natural symmetry on $\tilde V$, i.e.\ the map such that 
each $x \in V_1$ is mapped to the vertex in $V_2$ that corresponds to $x$ 
when $V_1$ and $V_2$ are identified with $V$. 
For $\tilde {\pi} \in S_{\tilde {V}}$, let  
$Q_A(\tilde {\pi})$ be the minimal $\tilde {\pi}$-invariant set that 
contains $A$ and is 
{\em compatible with} $\phi$. The latter means that $\phi(Q_A)=Q_A$. 
See also Figure \ref{Fig:decayboundary}.
\begin{figure}
\centering
\includegraphics[scale=0.25]{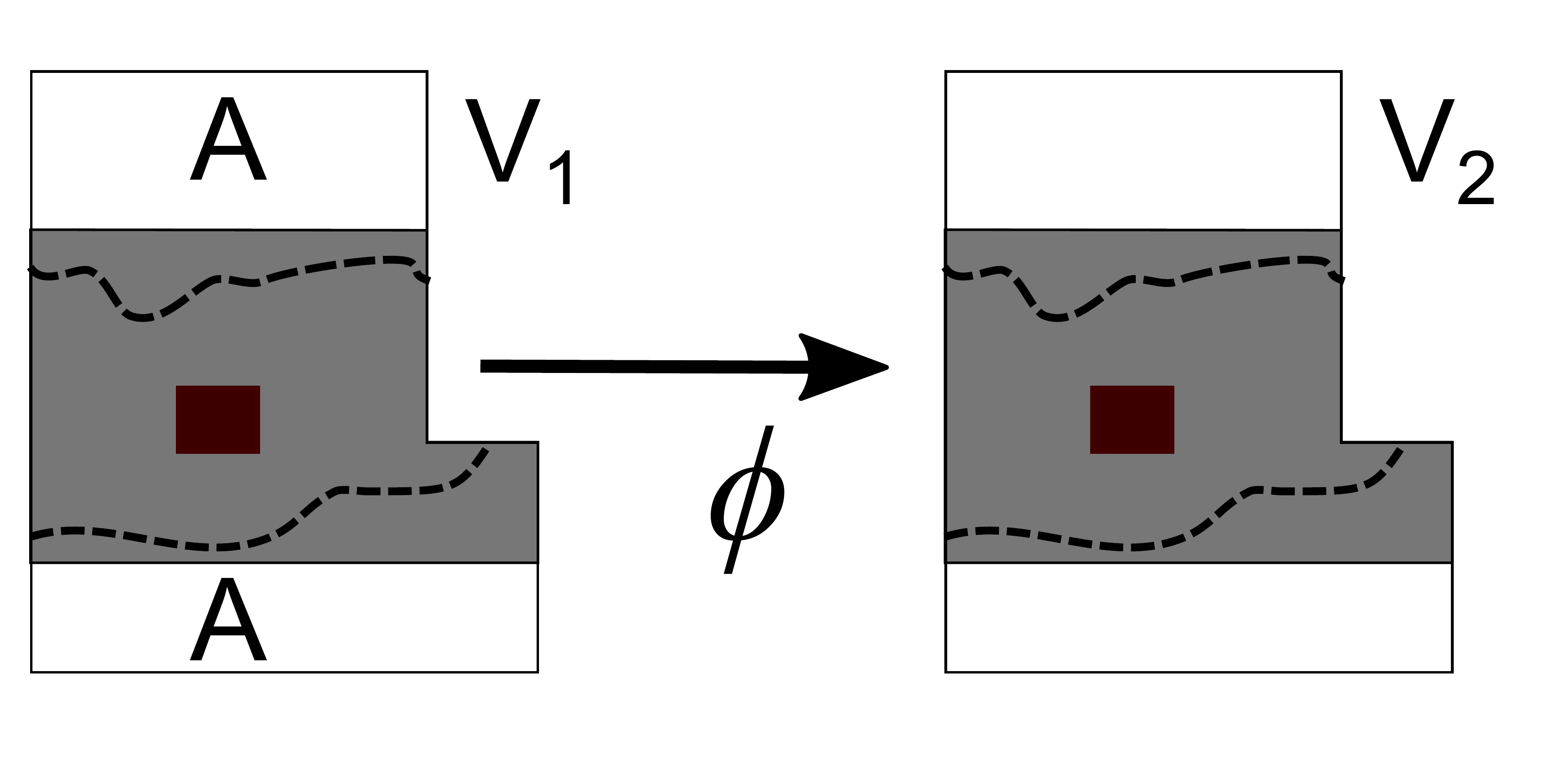}
\caption{The black set in $V_1$ corresponds to $B$, the dark set in $V_1$ corresponds to $U_1$,
the dark dotted line in $V_1$ and $V_2$ represents the  boundary of the (random) minimal $\phi$-compatible invariant set containing $A \subset V^1$.}
\label{Fig:decayboundary}
\end{figure}
Since the intersection of two $\tilde  \pi$-invariant, $\phi$-compatible sets 
containing $A$ retains these properties, and since $\tilde V$ has them, 
it is clear that such a minimal set exists.
\begin{proposition}
Under the assumptions made above,
\begin{equation}
	\label{eq:boundary condition decay}
	|\bbE_{U}(f) - \bbE_{V}(f)| \leq 2 \|f\|_{\infty} 
	\bbP_{\tilde V}\big(Q_A \cap B \neq \emptyset  \, \Big | \,  \tilde {\pi}|_A = {\rm id}\big).  
\end{equation}
\end{proposition}
\begin{proof}
As before, the superscript $\, \tilde{} \, $ will refer to objects that are defined
in $\tilde G$, the disjoint union of two copies of $G$.
Any graph permutation $\tilde {\pi}$ on $\tilde V$ can be written as 
$\pi_1 \oplus \pi_2$ with $\pi_i$ a permutation on $V_i$. 
For a 
$\mathcal{F}_B$-measurable random variable
$f$ on $S_V$, we define $f_1(\tilde {\pi}) = f(\pi_1)$ and 
$f_2(\tilde {\pi}) = f(\pi_2)$. We also define $A = V_1 \setminus U_1$,
where $U_1$ is $U$ regarded as 
a subset of $V_1$. In particular, $A \cap V_2 = \emptyset$.  
Then 
\begin{equation}\label{eq:expectationU}
\bbE_{U}(f) = \bbE_{\tilde V}(f_1 \, \, | \, \, \{ \tilde {\pi}|_A = {\rm id} \}), 
\end{equation}
and 
\[ 
\bbE_{V}(f) = \bbE_{\tilde V} (f_2) = \bbE_{\tilde V}(f_2 \, \, | \, \, \{ \tilde {\pi}|_A = {\rm id} \}).
\]
From the definition  $Q_A(\tilde {\pi}) $ we have that 
 $Q_A(\tilde {\pi}) = D$ if and only if \\
(i): $D \supset A$, $D$ is $\tilde {\pi}$-invariant and $\phi$-compatible;\\
(ii): Every strict subset of $D$ fails to have at least one of the three properties.\\
So $\{Q_A = D\}$ is $\caF_D$-measurable, and therefore $Q_A$ is an admissible random 
invariant set.  Now  (\ref{eq:expectationU}) gives,
\begin{multline}
\bbE_{\tilde V}(\mathbbm{1}\{\tilde {\pi}|_A = {\rm id}\}) \bbE_{U}(f) = 
\bbE_{\tilde V} (f_1 \mathbbm{1}\{\tilde {\pi}|_A = {\rm id}\}) = \\ \bbE_{\tilde V}\big(f_1 \mathbbm{1}\{\tilde {\pi}|_A = {\rm id}\} \mathbbm{1}\{Q_A \cap B = \emptyset\}\big) + 
\bbE_{\tilde V}\big(f_1 \mathbbm{1}\{\tilde {\pi}|_A = {\rm id}\} \mathbbm{1}\{Q_A \cap B \neq \emptyset\}\big),
\end{multline}
where $B$ is regarded as a subset of $V_1$.
The first term of the right hand side is equal to 
\[
\bbE_{\tilde V} \Big( \bbE_{\tilde V} ( f_1 \mathbbm{1}\{\tilde {\pi}|_A = {\rm id}\} \mathbbm{1}\{Q_A \cap B = \emptyset\} \Big| \caF_{Q_A} )\Big) = 
\bbE_{\tilde V} \Big( \bbE_{Q^c} \big (f_1 ( .  \oplus id)  \big) \mathbbm{1}\{\tilde {\pi}|_A = {\rm id}\}  \mathbbm{1}\{{Q_A} \cap B = \emptyset\}\Big),
\]
since $\mathbbm{1}\{\tilde {\pi}|_A = {\rm id}\}$ and $\mathbbm{1}\{{Q_A} \cap B = \emptyset\}$ are both 
$\caF_{Q_A}$-measurable, and thus Proposition \ref{prop:strong Markov} applies. But by the 
symmetry of ${Q_A}$, we have $\bbE_{Q_A}(f_1) = \bbE_{Q_A}(f_2)$. Thus we can do the same 
calculation with $f_2$ instead of $f_1$, and subtract the results. If in addition we
divide by $\bbE_{\tilde V}(\mathbbm{1}\{\tilde {\pi}|_A = {\rm id}\})$ 
we arrive at 
\begin{equation}
	\bbE_{U}(f) - \bbE_{V}(f) = \bbE_{\tilde V} \Big( (f_1 - f_2) 
	\mathbbm{1}\{{Q_A} \cap B \neq \emptyset\} \Big| \{ \tilde {\pi}|_A = {\rm id} \}\Big),
\end{equation}
which implies (\ref{eq:boundary condition decay}) and concludes the proof of the proposition.
\end{proof}

When we want to compare two different subsets $U_1$ and $U_2$ of $V$, 
we could slightly adapt the definition of $A$ above to get a similar estimate. 
In most cases, simply using the triangle inequality on 
\eqref{eq:boundary condition decay} is enough. 

In order to estimate the right-hand side of (\ref{eq:boundary condition decay}), we now need 
a tool to show the existence of invariant subsets with 
certain prescribed symmetries, and containing certain prescribed subsets of $V$. 
This tool is iterative sampling. 
Iterative sampling is a procedure to build a 
$\bbP_V$-distributed random variable step by step, reminiscent of the way that 
the full path of a Markov chain can be sampled step by step. In words, what we do is 
the following: we first sample a random permutation from $\bbP_V$, but keep only the 
cycles intersecting a (possibly random) set $K_1$. The only restriction is that $K_1$ may
not depend on the permutation we just sampled, it has to be chosen independently. We 
end up with a set $D_1$ where we have kept the cycles, and a set $B_1$ where we have 
discarded them, where $D_1$ and $B_1$ are disjoint and their union is $V$.
We then resample the permutation, independently, inside the subgraph 
generated by $B_1$. Again, we 
only keep the cycles intersecting some set $K_2$ that may depend on 
what has happened before and some external randomness, but not on the permutation inside
$B_1$. We carry on until we exhaust $V$. Formally, the definition is as follows.  

\begin{definition}
	\label{def:sampling strategy}
	A {\em sampling strategy} for spatial random permutations on a finite graph 
	$G = (V,E)$ consists of two families of random variables $(\sigma_A)_{A \subset V}$ 
	and $(K_A)_{A \subset V}$ such that
	\begin{itemize}
		\item [(i):] for all $A$, $\sigma_A$ is a $\bbP_A$-distributed random variable, in particular has values in $S_A$. $(\sigma_A)_{A \subset V}$ is a family of independent 
	random variables.
	\item[(ii):] $K_A$ takes values in the power set $\caP(A)$ for all $A$, and $\bbP(A = \emptyset) = 0$ when $A \neq \emptyset$.
	\item[(iii):] For each $A \subseteq V$, the random variable $K_A$ is independent of the family $(\sigma_B)_{B \subseteq A}$. 	
	\end{itemize}
\end{definition}

Note in particular that we require that $K_A$ is independent of $\sigma_A$. Given a 
sampling strategy, we define a \textit{recursive sampling procedure} as follows: We set 
$B_1 = V$, and for $i \geq 1$ we set 
\begin{equation} \label{eq:sampling procedure}
\pi_i = \sigma_{B_i}, \quad D_i = \Or_{\pi_i} (K_{B_i}), \quad B_{i+1} = B_i \setminus D_i.
\end{equation}
Above $\Or_\pi(A) = \{ \pi^j(x): x \in A, j \in \bbN \}$ denotes the orbit of $A$ under 
$\pi$. 
This way we end up with a sequence $(B_i, D_i, \pi_i)$ where $D_i \in {\rm Inv(\pi_i)}$,
and where the $D_i$ form a partition of $V$. It follows that the map 
\begin{equation}
	\label{eq:big permutation}
	\boldsymbol{\pi}: V \to V, \quad \boldsymbol{\pi}(x) = \pi_i(x) \text{ whenever } x \in D_i
\end{equation}
is a random element of $S_V$. 

\begin{lemma}\label{theo:samplingstrategy}
	For any sampling strategy, the distribution of $\boldsymbol{\pi}$ is $\bbP_V$.
\end{lemma}

\begin{proof}
	We prove the theorem by induction on the number of sets $K_i$. Let us call 
	$(K_A)_{A \subseteq V}$ an $n$-step sampling strategy if $K_A(\omega) = A$ 
	for all $\omega$ in the underlying probability space such that $A = B_{n+1}(\omega)$. 
	Since the event $\{A = B_{n+1}\}$ is independent of 
	$(\sigma_{\tilde A})_{\tilde A \subseteq A}$, any sampling strategy can be turned into
	an $n$-step sampling strategy, and the union over the $n$-step sampling strategies for 
	all $n$ gives all sampling strategies.   
	
	For a one-step sampling strategy, we have that $K_A = A$ for all $A \neq V$. Putting 
	$D_A(\bar \pi) = \Or_{\bar \pi}(A)$, we compute 
	\begin{equation} \label{eq:decomp}
	\bbP(\boldsymbol{\pi} = \bar \pi) = \sum_{A \subset V} \bbP(\boldsymbol{\pi} = 
	\bar \pi, K_V = A) = \sum_{A \subset V} \bbP(\sigma_V = \bar \pi |_{D_A(\bar \pi)}, 
	\sigma_{V \setminus D_A(\bar \pi)} = \bar \pi |_{V \setminus D_A(\bar \pi)}, K_V = A).
	\end{equation}
	For fixed $A$, all three events in the probability above are independent and moreover
	\[
	\bbP(\sigma_V = \bar \pi |_{D_A(\bar \pi)}) = \frac{1}{Z(V)} 
	\e{-\alpha \caH_{D_A(\bar \pi)}(\bar \pi)} Z \big( V \setminus D_A(\bar \pi) \big)
	\]
	 and 
	\[
	\bbP(\sigma_{V \setminus D_A(\bar\pi)} = \bar \pi |_{V \setminus D_A(\bar \pi)}) = \frac{1}{Z \big( V \setminus D_A(\bar \pi) \big) } 
	\e{-\alpha \caH_{V \setminus D_A(\bar \pi)}(\bar \pi)} .
	\]
	Since $\caH_{D_A(\bar \pi)}(\bar \pi) + \caH_{V \setminus D_A(\bar \pi)}(\bar \pi) = 
	\caH_V(\bar \pi)$, the product of the two terms above does not depend on $A$ and is 
	equal to $\bbP_V(\bar \pi)$, and summing \eqref{eq:decomp} over $A$ gives the result
	in the case of one-step sampling strategies. 
	
 	Now assume that the claim holds for all $n$-step sampling strategies up to some 
 	$n \in \bbN$. Let $(\sigma_A, K_A)_{A \subset V}$ 
 	be an $n+1$-step sampling strategy.  For fixed $\bar \pi \in S_V$ and 
 	$\bar K \subset V$ with $\bbP(K_V = \bar K, \sigma_V = \bar \pi) > 0$ we define the 
 	conditional measure 
 	\[
 	\bbP^{\bar \pi, \bar K} = \bbP( . | K_V = \bar K, \sigma_V = \bar \pi).
 	\]
 	We write $\bar D = \Or_{\bar \pi}(\bar K)$ and $\bar B = V \setminus \bar D$, and 
 	claim that under the measure $\bbP^{\bar \pi, \bar K}$, the families 
 	$(\sigma_A)_{A \subset \bar B}$ and $(K_A)_{A \subset \bar B}$ form a 
 	sampling strategy for random permutations on the graph generated by $\bar B$. 
 	To check this claim, note that the event 
 	$\{K_V = \bar K, \sigma_V = \bar \pi\}$ is independent of the family 
 	$(\sigma_A)_{A \subset \bar B}$, and thus for all choices of $\pi_A \in S_A$, 
 	\begin{equation}
 		\label{eq:first indep}
 		\bbP^{\bar \pi, \bar K}(\sigma_{A} = \pi_A \, \, \forall A \subset \bar B) = 
 		\bbP(\sigma_{A} = \pi_A \, \, \forall A \subset \bar B).
 	\end{equation}
 	This shows independence of the family $(\sigma_A)_{A \subset \bar B}$. Furthermore, 
 	for $A \subset \bar B$, $\kappa \subset A$, and $A_1, \ldots A_m \subseteq A$, we have 
 	\[
 	\bbP^{\bar \pi, \bar K}(K_A = \kappa, \sigma_{A_i} = \pi_{A_i} \, \forall i ) = 
 	\frac{\bbP(K_A = \kappa, K_V = \bar K, \sigma_V = \bar \pi, \sigma_{A_i} = \pi_i \, \forall i)}{\bbP(K_V = \bar K, \sigma_V = \bar \pi)} = \bbP^{\bar \pi, \bar K}
 	(K_A = \kappa) \prod_{i=1}^m \bbP(\sigma_{A_i} = \pi_i).
 	\]
 	Thus using \eqref{eq:first indep} from right to left, we see that $K_A$ is independent
 	of the family $(\sigma_{\tilde A})_{\tilde A \subset A}$ under 
 	$\bbP^{\bar \pi, \bar K}$, and thus we have indeed a sampling strategy, 
 	which in addition clearly is an $m-1$-step sampling strategy. Now, 
 	\[
 	\begin{split}
 	\bbP(\boldsymbol{\pi} & = \bar \pi) = \sum_{\bar K \subset V} \,\, \sum_{\bar \sigma \in 
 	S_V: \bar \sigma|_{\bar D} = \bar \pi|_{\bar D}} \bbP(\boldsymbol{\pi} |_{\bar B} = 
 	\bar \pi |_{\bar B}, K_V = \bar K, \sigma_V = \bar \sigma) = \\
 	& = \sum_{\bar K \subset V} \,\, \sum_{\bar \sigma \in 
 	S_V: \bar \sigma|_{\bar D} = \bar \pi|_{\bar D}} \bbP^{\bar \pi, \bar K}(\boldsymbol{\pi} = \bar \pi|_{\bar B}) \bbP(K_V = \bar K, \sigma_V = \bar \sigma) \\
 	& = \sum_{\bar K \subset V} \,\, \sum_{\bar \sigma \in 
 	S_V: \bar \sigma|_{\bar D} = \bar \pi|_{\bar D}} \bbP_{\bar B}(\bar \pi) \bbP(K_V = \bar K)
 	\e{-\alpha \caH_V(\bar \sigma)} \frac{1}{Z(V)}.
 	\end{split}
 	\]
 	In the last line we used the induction hypothesis, and the independence of $K_V$ and 
 	$\sigma_V$. For fixed $\bar K$, the sum over $\bar \sigma$ can now be carried out, 
 	and the resulting term is 
 	\[
 	\bbP_{\bar B}(\bar \pi) \bbP(K_V = \bar K) \e{-\alpha H_{\bar D}(\bar \pi)}  
 	\frac{Z(\bar B)}{Z(V)} = \bbP(K_V = \bar K) \bbP_V (\bar \pi). 
 	\]
 	The result now follows by summing over $\bar K$.
 	\end{proof}

Let us now come back to the task of constructing invariant sets with prescribed 
symmetries. We will be slightly more general than in the discussion leading up to 
\eqref{eq:boundary condition decay}, as this generality will be needed further on.
Let $G=(V,E)$ be a finite graph. 
A {\em symmetry} of $G$ is a bijective map $\phi$ on $V$ such that 
for all $x,y \in V$, $\{\phi(x),\phi(y)\} \in E $ if and only if $(x,y) \in E$. 
For a group $\Phi$ of symmetries on $G$ and an element $\pi \in S_V$, we say that 
$U \subset V$ is a $\Phi$-compatible $\pi$-invariant set if $U$ 
is invariant under $\pi$ as well as under all $\phi \in \Phi$. 

Let us first describe our sampling procedure in words. We start with the 
full graph $(V,E)$, $A \subset V$, and choose the symmetrization 
$\Phi(A) := \bigcup_{\phi \in \Phi} \phi(A)$ of $A$ as our
first set of which we want to keep the cycles. We draw a sample 
$\sigma_V$, and look at the cycles intersecting $K_1 := \Phi(A)$. If none 
of the cycles of $\sigma_V$ intersecting $K_1$ leaves $K_1$, this means
that $K_1$ is already a $\Phi$-compatible, $\boldsymbol{\pi}$-invariant 
subset of $V$, and we are done. 
Otherwise, we take $D_1 = {\rm Or}(K_1)$, and $B_2 = 
V \setminus D_1$. $B^c_2$ is in general not a $\Phi$-compatible subset of $V$:
given a point $x \in D_1 \setminus K_1 \subset B^c_2$, all of its images under maps
$\phi \in \Phi$ might be elements of $B_2$. Since $x$ is an element
of $B_2^c$, this means that $B_2^c$ is not $\Phi$-compatible in this case.
We define $\tilde K_2 = B_2 \cap \Phi(B_2^c)$ (see also Figure \ref{Fig:samplingprocedure}).
$\tilde K_2$ may still be 
empty, e.g.\ if we are lucky enough so that all cycles leaving $K_1$ jointly
cover a $\Phi$-compatible set. In that case, we are done, and $B^c_2$ is the 
desired set, as it is $\Phi$-compatible, $\pi$-invariant and contains $A$.
Otherwise, we put $K_2 = \tilde K_2$. 
By the considerations we have just 
made, it follows that the cardinality of $K_2$ is bounded by 
$( |\Phi| - 1 ) |D_1 \setminus K_1|$. This estimate will be useful later on. 
We now sample a random permutation $\sigma_{B_2}$ and keep all the cycles 
intersecting $K_2$ and arrive to a set $D_2$. Again, if $D_2 = K_2$, 
then $V \setminus K_2$ is a $\Phi$-compatible, $\boldsymbol{\pi}$-invariant
subset of $V$, and we are done. Otherwise, we repeat the procedure, i.e.\ 
we set $B_3 = B_2 \setminus D_2$ and $\tilde K_3 = B_3 \cap \Phi(B_3^c)$. 
\begin{figure}
\centering
\includegraphics[scale=0.25]{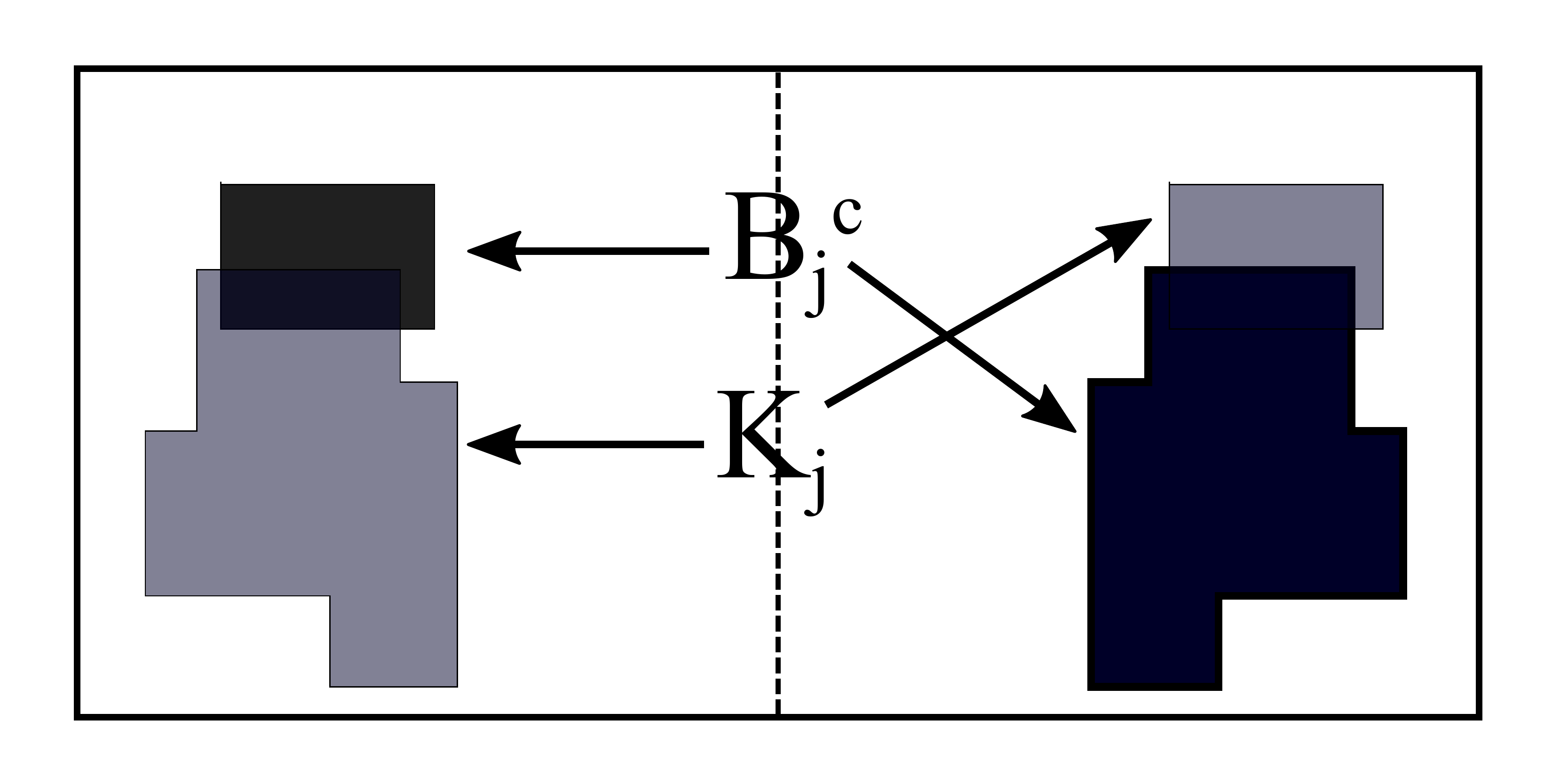}
\caption{When the group of symmetries $\Phi$ contains the identity and the reflection with respect to the dotted vertical line and $B_j^c$ corresponds to the black set in the figure, then $K_j$ corresponds to the grey set.}
\label{Fig:samplingprocedure}
\end{figure}
As before, we have $|K_3| \leq ( |\Phi| - 1) |D_2 \setminus K_2|$. We continue in 
this way until we exhaust $V$ (in which case we have found no non-trivial
$\Phi$-compatible subset), or until at some point $D_j = K_j$, in which case 
$B_{j+1}^c = \bigcup_{i \leq j} D_i$ is a $\Phi$-compatible, $\boldsymbol{\pi}$-invariant  set containing $A$.

Formally, the sampling strategy is the following. The family 
$(\sigma_B)_{B \subset V}$ are just independent random variables where 
$\sigma_B$ is $\bbP_B$-distributed. For given subset $A$, the maps 
$K_B$ are defined by  
\begin{equation}
	\label{eq:sampling strat 1}
	K_V = \Phi(A), \quad \text{ and for } B \neq V:  \quad  K_B = B \cap \Phi(B^c) \text{ if } 
	B \cap \Phi(B^c) \neq \emptyset, \text{ and } K_B = B 
	\text{ otherwise.} 
\end{equation}
It is obvious that this is indeed a sampling strategy. For any configuration
$\omega$ we put, 
\begin{equation}
\label{eq:hat A}	
N(\omega) = \inf \{n \in \bbN: K_{B_n(\omega)}(\omega) = B_n(\omega) \}, \text{  and }
\hat A(\omega) = B_{N(\omega)}^c (\omega).
\end{equation}
Then $\hat A(\omega)$ is a $\boldsymbol{\pi}(\omega)$-invariant subset, compatible 
with $\Phi$, and containing $A$. 
Recall that, from Lemma \ref{theo:samplingstrategy}, $\boldsymbol{\pi}$
is distributed like $\mathbb{P}_V$.

We say that $U$ separates $A \subset V$ from $B 
\subset V$ if $A \subset U$ and $B \cap U = \emptyset$. Thus, 
if in addition $\hat A(\omega) \cap B = \emptyset$ 
for some subset $B$, we have found a $\Phi$ compatible, 
$\boldsymbol{\pi}(\omega)$-invariant subset separating $A$ from $B$. We summarize: 

\begin{proposition}
	\label{prop:invarSubset1}
	For $A,B \subset V$ and $\hat A$ defined as in \eqref{eq:hat A}, we have 
	\[
	\bbP_V(\text{A $\pi$-invariant, $\Phi$-compatible subset separating $A$ from $B$ exists}) 
	\geq \bbP(\hat A \cap B = \emptyset).
	\]	
\end{proposition}

The previous proposition will only be 
useful if we have a way to control the probability that the 
random set $\hat A$ intersects $B$. At this point, we currently have no better tool than 
the following very crude estimate: if we write $d$ for the graph distance on $G$, then 
\begin{equation}
	\label{eq:observation}
| \hat A \setminus \Phi(A) | < d (\Phi(A),B) \quad \text{implies} \quad \hat A \cap B = \emptyset.
\end{equation}
Inequality (\ref{eq:observation}) holds true since by construction, from every $x \in \hat A$ a path in $G$ that is
entirely contained in $\hat A$ leads to 
$\Phi(A)$. Fortunately, this crude estimate will be sufficient for our purposes. 
We therefore now consider bounds on the cardinality of $\hat A$. 

We now introduce the notion of random permutation with
bounded cycle length. Even if so far in this section
we considered only finite graphs, we introduce this notion for graphs that might be finite
or infinite.
Recall that a random variable $X$ is {\em stochastically dominated} by another random 
variable $Y$ if $\bbP(X \geq k) \leq \bbP(Y \geq k)$ for all $k$. 
Let $\xi$ be a random variable on the positive integers. 
Consider a graph $(V,E)$ that might be finite 
or infinite.  
We say that our random permutation model has 
 \textit{cycle length bounded by
$\xi$} in $(V,E)$ if, uniformly in all finite sub-volumes $U \subset V$ and all points $x \in U$, 
the length of the cycle $\gamma_x$ containing $x$ is 
stochastically dominated by $\xi$. In other words, 
for all $\ell \geq 1$ we assume that  
\begin{equation}
\label{eq:bounddistributionP}
\sup_{\substack{ U \subset V \\ |U| < \infty } } \sup_{x \in U} \mathbb{P}_{U}( \,  |\gamma_x |\geq \ell ) 
\leq P( \xi \geq \ell).
\end{equation}
We use this notion in the next proposition.
\begin{proposition}
\label{prop:boundingorbit}
	Let $\xi$ be an integer-valued random variable, and 
	assume that a random permutation has cycle length bounded by $\xi$
	on the finite graph $(V,E)$.
	Let $(\xi_n)_{n \in \bbN}$ be a sequence of independent copies of $\xi$.
	Then for each $A \subset V$ and each $\ell \in \bbN$, we have 
	\[
	\bbP_V(|{\rm Or}(A)| \geq \ell) \leq \bbP(\sum_{i=1}^{|A|} \xi_i \geq \ell).
	\]	
\end{proposition}

\begin{proof}
	We consider the following sampling strategy: let $|A| = n$ and let 
	$x^1, \ldots, x^n$ be the elements of $A$. We sample the cycles 
	intersecting $A$ one by one, in the order of the $x^j$. In other 
	words, for $B \cap A \neq \emptyset$ we set $j_B =  \min \{i \leq n: x^i \in B \}$,
	$x(B) = x^{j_B}$, and $K_B = \{x(B)\}$. 
	When $B \cap A = \emptyset$, we set $K_B = B$. This way, we get a
	sampling procedure with sets $(B_i)_{i \in \bbN}$. Let 
	$L_i = |\gamma_{x(B_i)}|$ be the number of vertices in the $i-th$ cycle that 
	is sampled. When $A$ is exhausted after $\tau$ steps, we set $L_i=0$ for 
	all $i > \tau$. Note that $\tau(\omega) \leq n$ for any $\omega$,
	where $\omega$ denotes the realization in the probability space
	of the sampling procedure.
	The sequence $(L_{i})_{i \in \bbN}$ of random variables
	is definitely not a Markov chain, and usually $L_i$ will not even be 
	measurable with respect to the $\sigma$-algebra generated by all the 
	other $L_k$. However, it has the property that for given $\bar B \subset V$ 
	and given cycle $\bar \gamma$, 
	\begin{equation}
	\label{eq:prooflemmabounding1bis}
	\bbP(L_{k+1} > \ell | B_k = \bar B \text{ and } 
	\gamma_{x(\bar B)} = \bar \gamma) = 
	\bbP_{\bar B \setminus \bar \gamma}
	(|\gamma_{x(\bar B \setminus \bar \gamma)}| > \ell) \leq 
	\bbP(\xi_{k+1} > \ell). 
	\end{equation}
	Since $L_k = |\gamma_k|$, we conclude that when $\caF_k$ is the 
	$\sigma$-algebra generated by $L_1, \ldots, L_k$, then 
	$$\bbP(L_{k+1} > \ell | \caF_k)(\omega) \leq \bbP(\xi_{k+1} > \ell)$$
	for all $\omega$. We have just checked Assumption 2 of our comparison
	Lemma, which we give in the Appendix as Lemma \ref{lem:comparisonlemma}.
	In the notation of that lemma, $(M^1_j)_{j \geq 1} = (L_j)_{j \geq 1}$, and  
	$(M^2_j)_{j \geq 0} = (\xi_j)_{j \geq 0}$. It is easy to see that Assumptions 1 and 3 also hold.
	The second statement of Lemma \ref{lem:comparisonlemma} then implies that 
	for any $k \in \mathbb{N}$,
	\[
	\bbP(\sum_{j=1}^k L_j \geq \ell) \leq 
	\bbP(\sum_{j=1}^k \xi_j \geq \ell).
	\]
	The statement of our proposition now follows when observing that 
	$|{\rm Or}(A)| = \sum_{j=1}^{|A|} L_j$.
\end{proof}


\begin{proposition}
	\label{prop:size of invariant set}
	Consider a sampling strategy with the property that there exists $M \in \mathbb{N}$ such that under the recursive sampling procedure (\ref{eq:sampling procedure}), the inequality 
	$$
	| K_{B_{i+1}} | \leq M \, | D_i \setminus K_{B_{i}}	 |
	$$
	holds almost surely in the case $D_i \setminus K_{B_i} \neq \emptyset$, and $K_{B_{i+1}} = B_{i+1}$ otherwise.
	Let  $N$ be the random number of steps	such that $K_{B_N} = B_N$, and let $\hat A = B_N^c$ be the set sampled before  
	$N$ (see also  \eqref{eq:hat A}).
	Assume further that the random permutation has 
	cycle length bounded by an integer-valued random variable $\xi$ on $(V,E)$
	and let $(Z_j)_{j \geq 0}$ be a Galton-Watson process 
	where the offspring is distributed according to $M( \xi - 1)$, and with 
	initial population $M |K_V|$. Let $W$ be the (possibly infinite) total
	population of the Galton-Watson process. Then
	for each $\ell \in \bbN$, we have 
	\[
	\bbP_V(|\hat A| \geq \ell) \leq \bbP( W \geq  \frac{M}{M+1} \, \ell).
	\]	
\end{proposition}

\begin{proof}
By definition, $\hat A(\omega) = \bigcup_{i=1}^{N(\omega)-1} 
D_i(\omega)$. Moreover, $K_{B_i(\omega)}(\omega) \subset D_i(\omega)$ for all $i$, thus for all $n$,
\[
\bigcup_{i=1}^n D_i(\omega) = \bigcup_{i=1}^n \Big( D_i(\omega) \setminus K_{B_i(\omega)}(\omega) \Big) \cup 
\bigcup_{i=1}^n K_{B_i(\omega)}(\omega). 
\]
Whenever $K_{B_i(\omega)}(\omega) \neq B_i(\omega)$, we have assumed $|K_{B_{i+1}(\omega)}(\omega)| \leq M |D_i(\omega) \setminus K_{B_i(\omega)}(\omega)|$.
Therefore we find that for all $\omega$, 
\[
|\hat A(\omega)| \leq (M+1) \sum_{i=1}^{N(\omega) -1} 
|D_i(\omega) \setminus K_{B_i(\omega)}(\omega)|. 
\]
Write $q_i(\omega) = |D_i(\omega) \setminus K_{B_i(\omega)}(\omega)|$ for brevity. 
Then for given sets $\bar B$ and $\bar K$, 
\[
\bbP(q_j \geq \ell | B_j = \bar B, K_{B_j} = \bar K) = 
\bbP_{\bar B}({\rm Or}(\bar K) - |\bar K| \geq \ell) \leq 
\bbP(\sum_{i=1}^{|\bar K|} (\xi_i-1) \geq \ell),
\]
where in the last step we used Proposition \ref{prop:boundingorbit}.
Let $\caF_j$ be the sigma-algebra generated by the $(q_k)_{k \leq j}$.
Since $|K_{B_j(\omega)}(\omega)| \leq M \, q_{j-1}(\omega)$ for any $\omega$,
then we conclude that for any $\ell,m \in \mathbb{N}$,
for any integer  $j \geq 1$, and for any $\omega$ such that $q_{j-1}(\omega) = m$, 
\begin{equation}\label{eq:checkcomparisonlemma}
\bbP(q_j \geq \ell | \caF_{j-1})(\omega) \leq \bbP \Big( \, \sum\limits_{i=1}^{M m} (\xi_i - 1)   \geq \ell \, \Big) =
\bbP \Big( \, \sum\limits_{i=1}^{M m} M (\xi_i - 1)   \geq  M \ell \, \Big) \leq
\bbP\Big( Z_{j} \geq  M \ell  \Big|  Z_{j-1} =  m M \Big).
\end{equation}
Now we can again apply our comparison 
Lemma \ref{lem:comparisonlemma}. In the notation of that lemma,
$(M_j^1)_{j \geq 0} = (q_j)_{j \geq 0}$ and $(M_j^2)_{j \geq 0} = \frac{1}{M} (Z_j)_{j \geq 0}$,
which is the Galton-Watson process rescaled by a constant $M$.
Note that from Lemma \ref{lem:stochdomgalton}, the Markov chain
$(M_j^2)_{j \geq 0}$ satisfies the Assumption 1 of our comparison Lemma.
Note also that from (\ref{eq:checkcomparisonlemma}), we have that the
Assumption 2 of our comparison Lemma is fulfilled.
Furthermore, since  $\frac{1}{M} (Z_j)_{j \in \mathbb{N}}$ has initial population $\frac{1}{M} Z_0 = q_0 = |K_V|$, 
we also have that the Assumption 3 of the comparison Lemma is fulfilled.
Then, the Conclusion (b) of the comparison Lemma gives that,
$$
\bbP\Big(  \sum_{i=0}^{N(.)-1} q_i \geq \ell \Big) 
\leq \bbP \Big(  W \geq  M \ell\Big).
$$
Thus, we find that 
\[
\bbP( |\hat A| \geq \ell) \leq \bbP\Big( (M+1) \sum_{i=0}^{N(.)-1} q_i \geq \ell \Big) 
\leq \bbP \Big( (M+1) W \geq  M \ell\Big).
\]
The proof is finished. 
\end{proof}

Galton-Watson processes where the expected offspring per individual is strictly 
less than one  have exponentially bounded total population. We can use this for 

\begin{proposition}
	\label{prop:exponential bound} Consider the sampling strategy described 
	before Proposition \ref{prop:invarSubset1}. 
	Assume that the random permutation has cycle length bounded by $\xi$ on the finite graph $(V,E)$, and 
	that $(|\Phi| - 1) (\bbE(\xi) -1) < 1$. 
	Then there exist constants 
	$C_0>0$ and $\kappa_0 > 0$, depending on $\xi$ and $|A|$ but not on $V$, so that  
	\begin{equation}
	\label{eq:prop1}
	\bbP_V(|\hat A| \geq \ell) \leq C_0 \exp(- \kappa_0 \,  \ell),
	\end{equation}
	where $\hat A$ has been defined in \eqref{eq:hat A}.
	In particular, 
	\begin{equation}
		\label{eq:prop2}
	\bbP_V(\text{A $\pi$-invariant, $\Phi$-compatible subset separating $A$ from $B$ 
	exits}) \geq 1 - C_0\e{- \kappa_0 \left[ \, d(\Phi(A),B) + |\Phi(A)| \, \right]}.
	\end{equation}
	where $d(A,B)$ is the graph distance between $A$ and $B$. 
\end{proposition}

\begin{proof}
Observe that our sampling strategy (\ref{eq:sampling strat 1}) fulfils
$| K_{B_{i+1}}| \leq ( | \Phi| - 1 )  |D_i \setminus K_{B_i} |$ almost surely,
so Proposition \ref{prop:size of invariant set} can be applied with $M= |\Phi| - 1$.
Thus, the first statement follows from subcriticality of the Galton-Watson process.
For the second statement, we use 
that by \eqref{eq:observation}, 
\[
\bbP(\hat A \cap B = \emptyset) \geq 1 - \bbP\left[ |\hat A \setminus \Phi(A)| 
\geq d(\Phi(A),B)\right] = 1 - \bbP\left[|\hat A| \geq d(\Phi(A),B) + |\Phi(A)|\right]
. \]
\end{proof}

We will be interested in cases where the set $A$ itself is large. Then 
the above estimate is too weak, since the constant $C$ there will depend on $A$ 
we have too little control on its growth with $|A|$. We therefore need 
the following refinement, which we state only in the context of comparing
expectation of local function for different boundary conditions. An adaptation to 
the existence of $\Phi$-compatible sets is straightforward. 

\begin{proposition}
	\label{prop: decay of boundary condition 1}	Let $(V,E)$ be a finite graph, and $B \subset U \subset V$. Assume that 
	a random permutation on $(V,E)$ has cycle length bounded by a random variable $\xi$ 
	with $\bbE(\xi) < 2$. 
    Let $C>0$ and $\kappa>0$ be two constants such that the total
    offspring $W$ of a Galton-Watson process with initial
    population $1$ and offspring distribution $\xi-1$ fulfils $\bbP(W > n) \leq C \e{-2 \kappa n}$. 	
    Then, for any $\caF_B$-measurable function $f$ we have   
	\[
	|\bbE_V(f) - \bbE_U(f)| \leq 2 C \| f \|_{\infty} \sum_{x \in V \setminus U} 
	\e{- \kappa d (x, B)}.
	\]		
	Above, $d$ is the graph distance on $G$. 
\end{proposition}

\begin{proof}
Let $V_1$ and $V_2$ be two disjoint copies of $V$ and put 
$\tilde V = V_1 \cup V_2$, $\tilde V_0 = U_1 \cup V_2$ where $U_1$ is $U$ regarded as 
a subset of $V_1$, and $A = \tilde V \setminus \tilde V_0$. 
Let $\phi$ be the 
natural graph isomorphism taking $V_1$ to $V_2$. For $\tilde {\pi} \in S_{\tilde V}$, 
let $Q_A(\boldsymbol{\pi})$  be the minimal $\boldsymbol{\pi}$-invariant, 
$\phi$-compatible set containing $A$. In equation \eqref{eq:boundary condition decay}
we have seen that for $B \subset U$ and $\caF_B$-measurable $f$, 
\begin{equation}
	\label{eq:proof1}
|\bbE_V(f) - \bbE_U(f)| \leq 2 \| f \|_{\infty} \bbP_{\tilde V}
\Big( Q_A \cap B \neq \emptyset \Big|  \boldsymbol{\pi}|_A = {\rm id} \Big).
\end{equation}
On the right hand side above, we interpret $B$ as a subset of $U_1 \subset V_1$. 
See also Figure \ref{Fig:decayboundary}.

For estimating the right hand side, we first note that when $A,A' \subset \tilde V$, then 
\[
Q_{A \cup A'} = Q_A \cup Q_{A'}.
\] 
Indeed, the implication $\supset$ holds since $A \subset B$ implies $Q_A(\boldsymbol{\pi}) \subset Q_B(\boldsymbol{\pi})$
for any $\boldsymbol{\pi}$.
For the reverse implication, note that $Q_A(\boldsymbol{\pi}) \cup 
Q_{A'}(\boldsymbol{\pi})$ is $\boldsymbol{\pi}$-invariant and contains $A \cup A'$,
and since $\phi$ is bijective we also have $\phi(Q_A \cup Q_{A'}) = \phi(Q_A) \cup 
\phi(Q_{A'}) = Q_A \cup Q_{A'}$. So, $Q_A \cup Q_{A'} \supset Q_{A \cup A'}$ by the 
minimality of the latter set. We conclude
\begin{equation}
\label{eq:proof2}	
\bbP_{\tilde V} \Big( Q_A \cap B \neq \emptyset \Big|  \boldsymbol{\pi}|_A = {\rm id} \Big)
 = \bbP_{\tilde V} \Big( \bigcup_{x \in A} \{ Q_{\{x\}} \cap B \} \neq \emptyset \Big| 
 \boldsymbol{\pi}|_A = {\rm id} \Big) \leq \sum_{x \in A} \bbP_{\tilde V} 
 \Big( Q_{\{x\}} \cap B \neq \emptyset \Big| \boldsymbol{\pi}|_A = {\rm id} \Big).
\end{equation}
For estimating the latter probabilities, define $\phi_0(y)$ so that $\phi_0(y)=\phi(y)$ for $y \in U_1 \cup U_2$, and $\phi_0(y)=y$ if $y \in U_2^c$. Here, $U_2$ denotes $U$ regarded as a subset of $V_2$. 
Let $Q_{\{x\}}^0$ be the minimal $\boldsymbol{\pi}$-invariant 
set containing $x$ so that $\phi_0(Q_{\{x\}}^0(\boldsymbol{\boldsymbol{\pi}}) ) = Q_{\{x\}}^0 (\boldsymbol{\pi})$.
First note that
$\bbP_{\tilde V}(\boldsymbol{\pi}|_A = {\rm id} ) = Z( \tilde V_0 ) / Z(\tilde V)$, and thus 
\[
\bbP_{\tilde V} \Big( Q_{\{x\}} \cap B \neq \emptyset \Big| \boldsymbol{\pi}|_A = {\rm id} \Big)
= \bbP_{\tilde V_0} (Q_{\{x\}}^0 \cap B \neq \emptyset).
\]
The relevant sampling strategy is then given by 
\[
K_{\tilde V_0} = \{x\}, \quad K_B = B \cap \phi_0(B^c)  \text{ if }  B \cap \phi_0(B^c) 
\neq \emptyset, \quad K_B = B \text{ otherwise.}
\]
As before, we let $N = \inf \{j \in \bbN: K_{B_j} = B_j \}$ and $\hat A = B_N^c$. 
Then $\hat A(\boldsymbol{\pi})$ is $\phi_0$-compatible and $\boldsymbol{\pi}$-invariant,
so $Q_{\{x\}}^0(\boldsymbol{\pi}) \subset \hat A(\boldsymbol{\pi})$ for all 
$\boldsymbol{\pi}$. The sampling strategy fulfils the assumptions of Proposition 
\ref{prop:size of invariant set} with $M=1$, and therefore 
\[
\bbP_{\tilde V_0}(Q_{\{x\}}^0 \cap B \neq \emptyset) \leq \bbP_{\tilde V_0} 
( | \hat A | \geq d(x,B) + 1) \leq C \e{-\kappa (d(x,B) + 1)}.
\]
$C$ and $\kappa$ are the constants defined in the statement of the proposition. 
Inserting this estimate in \eqref{eq:proof2}
and the result into \eqref{eq:proof1} shows the claim. 
\end{proof}

Consider a  vertex transitive graph $G = (V,E)$ that might be finite or infinite.
Theorem \ref{theo1} guarantees that there exists $\alpha_0' < \infty$ (which
may be significantly larger than the $\alpha_0$ given in that theorem) such
that for all $\alpha > \alpha'_0$, the corresponding random permutation has 
cycle length bounded by a random variable $\xi$ that
has finite exponential moments and has $\bbE(\xi) < 2$. Proposition 
\ref{prop: decay of boundary condition 1} then gives a bound on
$|\bbP_{V_1}(\pi|_A = {\rm id}) - \bbP_{V_0}(\pi|_A = {\rm id})|$ 
for finite subsets 
$A \subset V_1 \subset V_0$ of $V$. However, this bound becomes meaningless when 
$A$ is allowed to become large, since both of the involved probabilities then
become very small. The following proposition improves the estimate of 
Proposition \ref{prop: decay of boundary condition 1} to the strength that 
will be needed in the next section. We formulate it in terms of partition 
functions.

\begin{proposition}
	\label{prop:decay of boundary 2}
	Let $(V,E)$ be finite or infinite.
	Consider finite subsets $A \subset V_1 \subset V_0$ of $V$. Let $\alpha_0'$
	be large enough so that for all $\alpha > \alpha_0'$, the random permutation
	has cycle length bounded by a random variable $\xi$ with $E( \, \xi \, )<2$.
    Let $C>0$ and $\kappa>0$ be two constants such that the total
    offspring $W$ of a Galton-Watson process with initial
    population $1$ and offspring distribution $\xi-1$ fulfils $\bbP(W > n) \leq C \e{-2 \kappa n}$. 
    Let $\hyperlink{c_1}{c_1(\alpha)}$ be defined as in 
	\eqref{eq:c1}. Define 
	\begin{equation}
	\label{eq:D}	
	D = D(A,V_0,V_1,\alpha,\xi) = \exp \Big( \frac{ 2 C}{\hyperlink{c_1}{c_1(\alpha)}} 
	\sum_{x \in A} \sum_{y \in V_0 \setminus V_1} 
	\e{- \kappa d(x,y)} \Big),		
	\end{equation}	 	
	where $d(x,y)$ denotes the graph distance between $x$ and $y$. 
	Then 
	\[
	\frac{1}{D} \frac{Z(V_1 \setminus A)}{Z(V_1)} \leq
	\frac{Z(V_0 \setminus A)}{Z(V_0)} \leq 
	D \frac{Z(V_1 \setminus A)}{Z(V_1)}.
	\]		
\end{proposition}

\begin{proof}
Let us put $n = |A|$, and write 
$$
A = \{x^1, x^2, \ldots x^{n} \}, \quad \text{ and } \quad  
A_i = \{ x^i, x^{i+1}, \ldots, x^{n}\} \subset A$$
for $i \leq |A|$. We set $A_{n+1}=\emptyset$. Then, for $j = 0,1$, 
\[
\frac{Z(V_j \setminus A)}{Z(V_j)} = \prod_{i=1}^n \frac{Z(V_j \setminus A_i)}
{Z(V_j \setminus A_{i+1})},
\]
and we have 
\[
\log \frac{Z(V_1 \setminus A)}{Z(V_1)} = 
\sum_{i=1}^n \log \frac{Z(V_1 \setminus A_i)}{Z(V_1 \setminus A_{i+1})} 
= \sum_{i=1}^n \log \frac{Z(V_0 \setminus A_i)}{Z(V_0 \setminus A_{i+1})} 
+ \sum_{i=1}^n \log \Big( \frac{Z(V_1 \setminus A_i)}{Z(V_1 \setminus A_{i+1})}
\frac{Z(V_0 \setminus A_{i+1})}{Z(V_0 \setminus A_i)} \Big).
\]
The first sum on the right hand side above is equal to 
$\log \frac{Z(V_0 \setminus A)}{Z(V_0)}$. We use the inequality 
$\log (1 + |x|) \leq |x|$ on each term in the final sum above, and find 
\[
\Big |  \log \frac{Z(V_1 \setminus A)}{Z(V_1)} 
- \log \frac{Z(V_0 \setminus A)}{Z(V_0)} \Big |  \, \leq \, 
\sum_{i=1}^n   \Big | \, \, \frac{Z(V_1 \setminus A_i)}{Z(V_1 \setminus A_{i+1})} 
- \frac{Z(V_0 \setminus A_i)}{Z(V_0 \setminus A_{i+1})} \Big |   \, \, \,
\frac{Z(V_0 \setminus A_{i+1})}{Z(V_0 \setminus A_{i})}.
\]
Now, by applying Proposition \ref{prop:lowerboundpart} with 
$U = V_0 \setminus A_{i+1}$ and $A = \{ x^i \}$, we get 
\[
\Big | \frac{Z(V_0 \setminus A_{i+1})}{Z(V_0 \setminus A_{i})} \Big| \leq 
\frac{1} {\hyperlink{c_1}{c_1(\alpha)}}.
\]
%
%
%
We now apply Proposition \ref{prop: decay of boundary condition 1} with $f = \mathbbm{1}\{\pi(x^i)=x^i\}$
and we obtain that,
\begin{equation}
\label{eq:theIneq}	
\Big| \frac{Z(V_1 \setminus A_i)}{Z(V_1 \setminus A_{i+1})} - 
\frac{Z(V_0 \setminus A_i)}{Z(V_0 \setminus A_{i+1})} \Big| 
= | \bbP_{V_1 \setminus A_{i+1}}(\pi(x^i) = x^i) - 
\bbP_{V_0 \setminus A_{i+1}}(\pi(x^i) = x^i) |
\leq \sum_{y \in V_0 \setminus V_1} 2 C \e{-\kappa d(y,x^i)}.
\end{equation}
We combine these estimates and obtain
\[
\Big |   \, \log \frac{Z(V_1 \setminus A)}{Z(V_1)} - \log \frac{Z(V_0 \setminus A)}
{Z(V_0)} \Big  | \,  \leq \frac{ 2 \, C}{\hyperlink{c_1}{c_1(\alpha)}} 
\sum_{x \in A} \sum_{y \in V_0 \setminus V_1} \e{-\kappa d(x,y)}.
\]
Exponentiating this inequality gives $Z(V_1 \setminus A)/Z(V_1) 
\leq D Z(V_0 \setminus A)/Z(V_0)$ and thus the first claimed inequality. 
For the second one, notice that the only place where we used $V_1 \subset V_0$
is inequality \eqref{eq:theIneq}. That estimate 
still holds when interchanging the roles of $V_0$ and $V_1$, and when 
we do this, we get the remaining inequality. 
\end{proof}

We will only use the following weaker form of the previous proposition.
\begin{corollary}
	\label{cor:decay3}
	In the situation of Proposition \ref{prop:decay of boundary 2}, write 
	$B = V_0 \setminus V_1$. We then have
	\[
	\exp \Big( - \frac{ 2 C}{\hyperlink{c_1}{c_1(\alpha)}} |A| |B| \e{-\kappa d(A,B)} \Big)
	\frac{Z(V_1 \setminus A)}{Z(V_1)} \leq \frac{Z(V_0 \setminus A)}{Z(V_0)}
	\leq \exp \Big(  \frac{2 C}{\hyperlink{c_1}{c_1(\alpha)}} |A| |B| \e{-\kappa d(A,B)} \Big)
	\frac{Z(V_1 \setminus A)}{Z(V_1)}.
	\]
\end{corollary}

\rhead{ \large{ \textit{5 \hspace{0.7cm} PROOF OF THEOREM 2.2}}}
\section{Proof of Theorem \ref{theo2}}
\label{sect:proofoftheo2}

Recall that for $y = (y_1, \ldots, y_d) \in \bbZ^d$ we write $\fc y = y_1$
and $\sc y = (y_2, \ldots, y_d)$. In the proof we will need the following notions.
For $y \in \Lambda_n$, we write 
\[
\caC_{y} := \{x \in \Lambda_n: \fc x - \fc y \geq | \sc x - \sc y | \} 
\cup \{ x \in \Lambda_n: \fc x \geq \fc y + \log n \}
\]
for the forward cone starting in $y$ that is broadened to full width after 
a logarithmic length (see also Figure \ref{Fig:regstructures-a}). A set $A \subset \Lambda_n$ will be called 
\[ 
\begin{split}
& \text{ weakly $y$-admissible if } 
\{ x \in \Lambda_n: \fc x \geq \fc y + \log n \} \subset A, \\
&\text{ $y$-admissible if in addition }
\{ x \in \Lambda_n: \fc x \geq \fc y, \sc x = \sc y \} \subset A,
\\ 
&\text{ and  strictly $y$-admissible if in addition } 
\caC_y \subset A.
\end{split}
\]
We write $\caA_y^w$, 
$\caA_y$ and $\caA_y^s$ for the set of weakly admissible, admissible and 
strictly admissible sets, respectively.
For $y,z \in \Lambda_n$, $A \in \caA_y$ and 
$\pi \in \caS_A^{y \to z}$ 
recall that $\gamma(\pi) = \Orb_\pi(\{y\})$ denotes the trace of the 
self avoiding walk that is embedded in $\pi$, 
starting from $y$ and ending in $z$. We order the elements 
of $\gamma(\pi)$ by order of their appearance in $\Orb_\pi(\{y\})$, i.e.\ 
$x \leq y$ if $y \in \Orb_\pi(\{x\})$. Together with this order, the set 
$\gamma(\pi)$ uniquely characterises a self-avoiding path from $y$ to 
an element of $z$, and we will sometimes write $\caL_A^{y \to z}$ 
for the set of all ordered subsets of $A$ such that their order makes them 
a self-avoiding nearest neighbour walk in $A$ starting in $y$ and ending in $z$. 
As for permutations, we define $\caL_A^{y \to \ell_n} = \bigcup_{z \in \ell_n} 
\caL_A^{y,z}$.

Given $\gamma \in \caL_A^{y \to \ell_n}$, a point 
$x \in \gamma$ will be called {\em pre-regeneration point} of $\gamma$ if 
\[
z \in \caC_x \text{ for all } z \geq x, \qquad \text{and } 
\fc z < \fc x \text{ for }  z < x.\
\]
In other words, the self avoiding path hits the set $\ell_{\fc x}$ precisely 
at $x$ and stays in $\caC_x$ thereafter (see also Figure \ref{Fig:regstructures-b}. Note that the latter requirement is more 
than what is usually required for regeneration points of self-avoiding walks. We 
will need it below in order to decouple the future of $\gamma$ from part of 
the background consisting of loops in $\pi$. 

\begin{figure}
\centering
 \centering
    \begin{subfigure}{0.3\textwidth}
\includegraphics[scale=0.35]{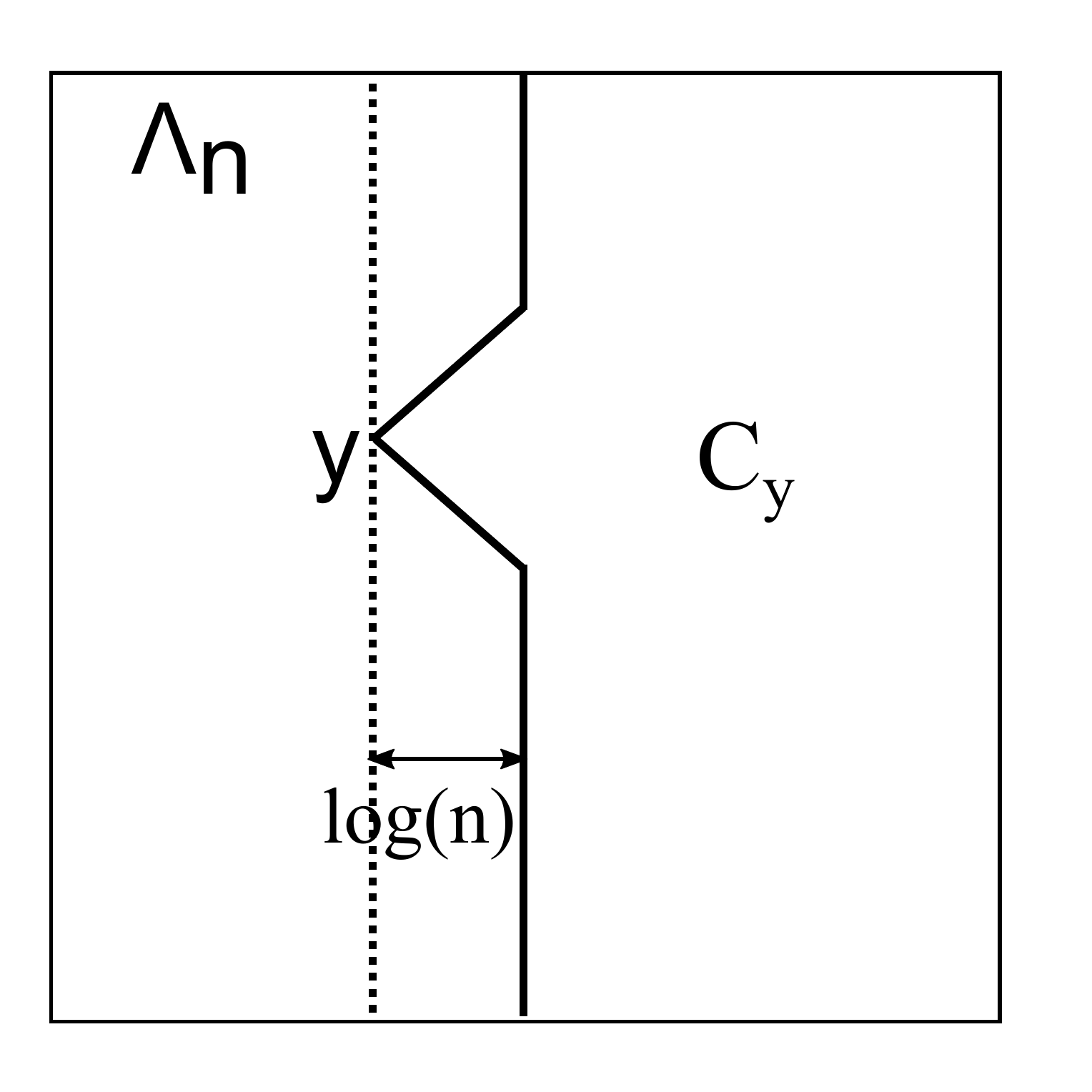}
\caption{}
\label{Fig:regstructures-a}
\end{subfigure}
\begin{subfigure}{0.3\textwidth}
\includegraphics[scale=0.35]{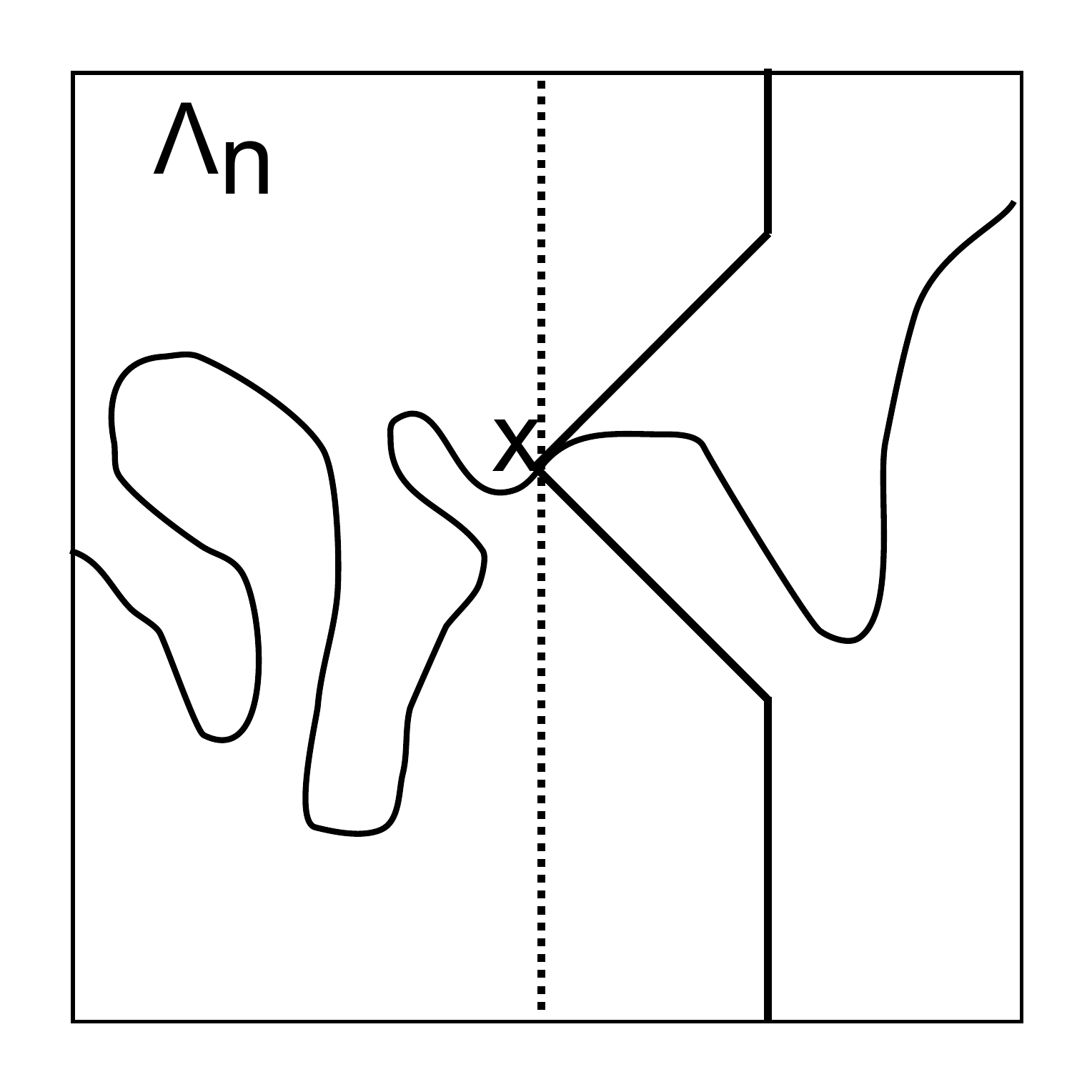}
\caption{}
\label{Fig:regstructures-b}
\end{subfigure}
\begin{subfigure}{0.3\textwidth}
\includegraphics[scale=0.33]{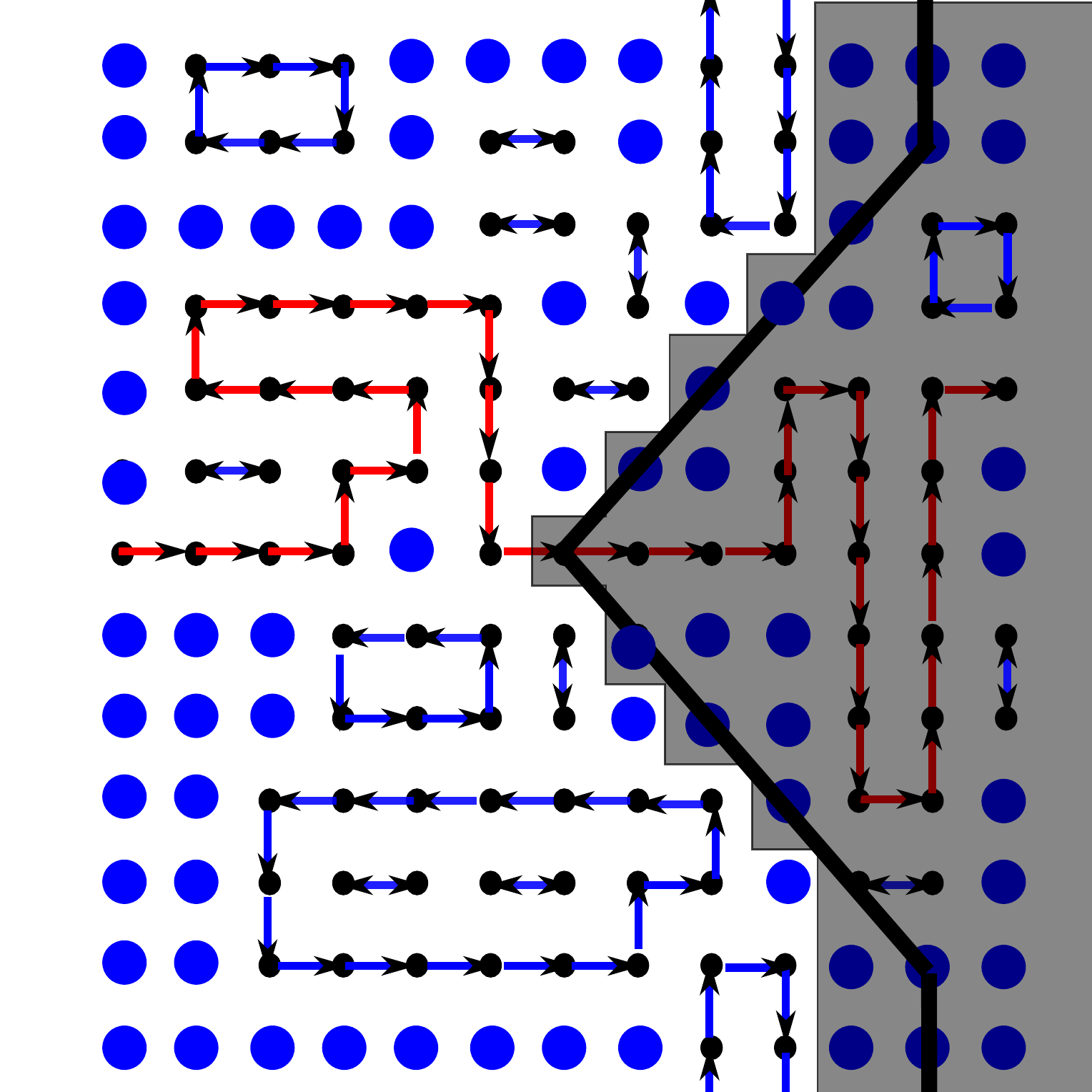}
\caption{}
\label{Fig:SAW-a}
\end{subfigure}
\caption{\textbf{(a)} A forward cone $C_y$. \textbf{(b)} A pre-regeneration point $x$. \textbf{(c)} The dark set represents the regeneration set of $\pi$ through $x$, where $x$ is the vertex in the centre of the figure.}
\label{Fig:regstructures}
\end{figure}

Let $\pi \in \caS_A^{y \to \ell_n}$, and let $\pi_0$ be the element
of $\caS_A$ with $\pi_0(x) = x$ whenever $x \in \gamma(\pi)$, and 
$\pi_0(x) = \pi(x)$ otherwise. 
A pre-regeneration point 
$x \in \gamma(\pi)$ will be called {\em regeneration point} if there exists 
a subset $R \subset A$ with the following four properties: 
\begin{itemize}
\item[(R1)] $R \subset \{z \in A: \fc z \geq \fc x \}$. 
\item[(R2)] $R$ is strictly $x$-admissible.
\item[(R3)] $R$ is invariant under reflections leaving 
$\{z \in A: \sc z = \sc x\}$ invariant. In other words, 
$(\fc z, \sc x - \sc z) \in R$ whenever $z \in R$. 
\item[(R4)] $R$ is invariant under $\pi_0$. 
\end{itemize}
If $x$ is a regeneration point of $\pi$, the largest set with properties
(R1)-(R4) is called {\em regeneration set} of $\pi$ through $x$, and is denoted 
by $R(x,\pi)$. See also Figure \ref{Fig:SAW-a}.
Since (R1)-(R4) are invariant under unions of sets, 
the existence of any regeneration set implies the existence and uniqueness of 
$R(x,\pi)$. By convention, the point $0 \in \Lambda_n$ will be a regeneration
point with $A(0,\pi) = \Lambda_n$ for all $\pi$, 
even though $0$ might not be a pre-regeneration point of $\gamma(\pi)$. 
Also, $z = \max \{ x: x \in \gamma(\pi) \}$ is a regeneration point
by convention, with empty regeneration set. The set of all regeneration points
of some $\pi \in \caS_A^{y \to \ell_n}$ is denoted by $\caR(\pi)$ and is 
non-empty by the above conventions. We order its elements 
by the order in which they appear in $\gamma$.

We now define for each 
$\pi \in \caS_{\Lambda_n}^{0 \to \ell_n}$ the following sequence of
pairs $(x_i,R_i) \subset \Lambda_n \times \caP(\Lambda_n)$: We set $X_0 = 0$ 
and $R_0 = \Lambda_n$, and then recursively define 
\[
X_{i+1}(\pi) = \begin{cases}
	\min \{y \in \caR(\pi): \fc y \geq \fc x + \log n, \} & \text{ if } 
	\fc X_i(\pi) + \log n \leq n,\\
	\max \{x: x \in \gamma(\pi) \} & \text{otherwise.}
\end{cases}
\]
In words, we pick a sequence of regeneration points that have mutual 
horizontal distance of at least $\log n$, when in this way we find a regeneration
point with horizontal distance less than $\log n$ to $\ell_n$, we jump to 
$\max \gamma$ and stay there forever. In addition, we define 
$R_i(\pi) = R(X_i(\pi), \pi)$. 

We claim that the sequence $(X_i, R(X_i))$ is a $\Lambda_n \times 
\caP(\Lambda_n)$-valued Markov chain, with transition matrix 
\[
p \big( (0,\Lambda_n),(x', R')\big) = \bbP_{\Lambda_n}^{0 \to \ell_n} 
(X_1 = x', R_1 = R'), 
\]
$p(\big (x,R),(y,\emptyset)) = \bbP_R^{x \to \ell_n}(y = \max \gamma)$ 
if $\fc x + \log n \geq n$, and 
\begin{equation}\label{MC transition matrix} 
\begin{split}
p\big( (x,R), (x',R') \big) & = \bbP_R^{x \to \ell_n} \big(
X_1 = x', R_1 = R' \big| \gamma \subset \caC_x \big)  \\ 
& =  \bbP_R^{x \to \ell_n}  
\big(x'=\min \{ y \in \caR: \fc y \geq \fc x + \log n\}, R' = R(x') \big| 
\gamma \subset \caC_x \big)
\end{split}
\end{equation} 
in all other cases. 
This will follow from a variant of the spatial Markov
property that we are now going to state. 

For any set $A \subset V = \Lambda_n$, we define $\caF_A$ to be the 
$\sigma$-algebra over $\caS_V^{0 \to \ell_n}$ generated by the forward 
evaluations $\pi \mapsto \pi(z)$ for $z \in A$. Note that in contrast to 
the situation in Proposition \ref{prop:Markovspatial}, we do not consider 
inverse images since these may not be defined in the situation with an 
open cycle. For any $\pi \in S^{0 \rightarrow \ell_n}_V$, we define 
the family of $\pi$-almost invariant sets
\[
{\rm Inv}_0(\pi) := \big \{A \subset V: \pi(A) \subset A \}.
\]
Note that by this definition, the open cycle $\gamma$ can enter an 
almost-invariant set $A$ but not leave it, and if 
$A \cap \gamma \neq \emptyset$ then the image of $A$ is 
$A$ without the unique entrance point of $\gamma$ into $A$. 
The reader should be warned that in the current situation, $A \in 
{\rm Inv}_0(\pi)$ does not imply $A^c \in {\rm Inv}_0(\pi)$.
Let $A \subset V$, $x \in A$, $a \in A^c$, 
and $\pi \in \caS_V^{a \to \ell_n}$ such that 
$A \in {\rm Inv}_0(\pi)$ and $x = \min \gamma(\pi) \cap A$. Then
there exist (unique) $\pi|_A \in \caS_A^{x \to \ell_n}$ and 
$\pi|_{A^c} \in \caS_{A^c \cup \{x\}}^{a \to x}$ 
such that $\pi|_A(z)= \pi(z)$ for all $z \in A$ and $\pi_{A^c}(z) = 
\pi(z)$ for all $z \in A^c$. In this situation, for $\caF_{A^c}$-measurable 
$f: \caS_V^{0 \to \ell_n} \to 
\bbR$, we use the same symbol to denote the function 
$f: \caS_{A^c \cup \{x\}}^{a \to x} \to \bbR$ with 
$f(\pi) = f(\pi_{A^c})$. In the same 
way, we get $g(\pi) = g(\pi_{A})$ for $\caF_{A}$-measurable $g$.

\begin{proposition}
\label{prop:Markovspatial open cycle}
Let $A \subset V$, $a \in A^c$, $x \in A$. Then, \\[3mm]
(i): $\bbP^{a \rightarrow \ell_n}_V(A \in {\rm Inv}_0, 
x = \min \gamma \cap A) = Z^{a \rightarrow x}(A^c \cup \{x\}) 
Z^{x \rightarrow \ell_n}(A) / Z^{a \rightarrow z}(V)$.\\[2mm]
(ii): For $\caF_{A^c}$-measurable $f$
and $\caF_{A}$-measurable $g$,
we have 
\[
\bbE^{a \rightarrow \ell_n}_V(f | A \in {\rm Inv}_0, x = \min \gamma \cap 
A) = \bbE^{a \rightarrow x}_{A^c \cup \{x\}} (f) 
\]
and 
\[
\bbE^{a \rightarrow \ell_n}_V(g| A \in {\rm Inv}_0, x = \min \gamma \cap A) 
= \bbE^{x \rightarrow \ell_n}_{A}(g).
\]
(iii): For  $\caF_{A^c}$-measurable $f$
and  $\caF_{A}$-measurable $g$,
we have 
\[
\bbE^{a \rightarrow \ell_n}_V(f \, g | A \in {\rm Inv}_0, x = \min \gamma \cap A) = \bbE^{a \rightarrow x}_{A^c \cup \{x\}}(f) \,  
\bbE^{x \rightarrow \ell_n}_{A}(g).
\]
(iv): For $\caF_{A}$-measurable $g$ and $\mathcal{Q} \in \caF_{A^c}$, we have 
\begin{equation}
\label{eq:spatialMarkov2open}
\bbE^{a \rightarrow \ell_n}_V(g | A \in {\rm Inv}_0, x = \min \gamma \cap A, \mathcal{Q}) = 
\bbE^{x \rightarrow \ell_n}_{A}(g).
\end{equation}
\end{proposition}
\begin{proof}
Let  
$\caH_B(\pi) := \sum_{x \in B} |\pi(x) - x|$. Then for all $A \subset 
V$, all $x \in A$ and all $\pi \in \caS_V^{a \to \ell_n}$ such that 
$A \in {\rm Inv}_0(\pi)$ and $x = \min \gamma \cap A$, 
we have 
\[
\caH_V(\pi) = \caH_{A^c \cup \{x\}}(\pi|_{A^c}) + \caH_A(\pi|_A).
\]
This holds since $\pi|_{A^c}(x) = x$. Consequently, for $\caF_{A^c}$ 
measurable $f$ and $\caF_A$-measurable $g$, we have 
\[
	\sum_{\substack{\pi \in S^{a \rightarrow \ell_n}_V: \\ 
	A \in { \rm Inv}_{p}(\pi), x = \min \gamma(\pi) \cap A} } 
	f(\pi) g(\pi)\e{- \alpha \mathcal{H}_{V}(\pi)} = 
	\sum_{\pi \in S^{a \rightarrow x}_{A^c \cup \{x\}}} f(\pi|_{A^c}) 
	\e{- \alpha \mathcal{H}_{A^c \cup \{x\}}(\pi_{A^c})} 
	\sum_{\tilde \pi \in S^{x \rightarrow \ell_n}_{A}} g({\tilde \pi_A})
	\e{-\alpha \mathcal{H}_{A}(\pi|_A)}.
\]
The proof now follows the same steps as  the proof of Proposition 
\ref{prop:Markovspatial}.
\end{proof}

For proving equation \eqref{MC transition matrix}, we fix $x_1, \ldots x_{k+1} \in 
\Lambda_n$ and subsets $Q_1, \ldots, Q_{k+1}$ of $\Lambda_n$ which satisfy 
the deterministic conditions for regeneration points, in particular 
for all $i \leq k+1$,  
\begin{equation} \label{inclusion}
\fc x_{i} \geq \fc x_{i-1} + \log n, \qquad \text{ and } \qquad  \caC_{x_i} \subset Q_i \subset \{ z: \fc z \geq \fc x_i \}.
\end{equation}
Let us write 
$\caJ_i = \{ X_i = x_i, R_i = Q_i \}$ for brevity, and write 
\[
C(x,y) = \{\pi:  y = \min \{ w \in \caR(\pi): \bar w \geq \bar x + \log n \} \} .
\]
Then 
\[
\begin{split}
& \bbP_{\Lambda_n}^{0 \to \ell_n} \Big( X_{k+1} = x_{k+1}, 
R_{k+1} = Q_{k+1} \Big| \bigcap_{i=1}^k \caJ_i \Big) =
\bbP_{\Lambda_n}^{0 \to \ell_n} \Big( C(x_k,x_{k+1}), Q_{k+1} = R(x_{k+1}) \Big| 
\bigcap_{i=1}^k \caJ_i \Big) \\
& = \bbP_{\Lambda_n}^{0 \to \ell_n} \Big( C(x_k,x_{k+1}), R(x_{k+1}) =Q_{k+1} 
  \Big| 
\Orb(x_k) \subset \caC_{x_k}, \gamma \setminus \Orb(x_k) \subset 
\{z: \fc z < \fc{x}_k \}, R_k = Q_k, \bigcap_{i=1}^{k-1} \caJ_i\Big)\\ 
& = 
\frac{\bbP_{\Lambda_n}^{0 \to \ell_n} \Big( C(x_k,x_{k+1}), R(x_{k+1}) = 
Q_{k+1},  
\Orb(x_k) \subset \caC_{x_k} \Big| \gamma \setminus \Orb(x_k) \subset 
\{z: \fc z < \fc{x}_k \}, R_k = Q_k, \bigcap_{i=1}^{k-1} \caJ_i\Big) } 
{\bbP_{\Lambda_n}^{0 \to \ell_n} \Big( 
\Orb(x_k) \subset \caC_{x_k} \Big| \gamma \setminus \Orb(x_k) \subset 
\{z: \fc z < \fc x_k \}, R_k = Q_k, \bigcap_{i=1}^{k-1} \caJ_i\Big) }.
\end{split}
\]
The event 
$C(x_k,x_{k+1}) \cup \{ R_{k+1} = R', \Orb(x_k) \subset \caC_{x_k}\}$ is 
$\caF_{\caC_{x_k}}$-measurable, and since $\caC_{x_k} \subset Q_k$, it is 
$\caF_{Q_k}$-measurable. On the other hand, the event 
\[
\caQ := \{ \gamma \setminus \Orb(x_k) \subset 
\{z: \fc z < \fc x_k \}, R_k = Q_k, \bigcap_{i=1}^{k-1} \caJ_i \}
\]
is $\caF_{Q_k^c \setminus \{x_k\}}$-measurable, and contains the set 
$\{ Q_k \in {\rm Inv}_{x_k} \} \cup \{x_k = \min \gamma \cap Q_k\}$. 
We can therefore apply part (iv) 
of Proposition \ref{prop:Markovspatial open cycle} with $A = Q_k$ to both 
numerator and denominator above, and find 
\[
\bbP_{\Lambda_n}^{0 \to \ell_n} \Big( X_{k+1} = x', R_{k+1} = Q_{k+1} \Big| 
\bigcap_{i=1}^k \caJ_i \Big) = \bbP_{Q_k}^{x_k \to \ell_n}\Big(X_1 = x_{k+1}, 
R_1 = Q_{k+1} \Big| \gamma \in \caC_{x_k} \Big),
\]
as claimed. 

By the symmetry condition (R3) of regeneration sets, the process 
$(\sc X_i)$, while not being a Markov process, is a $\bbZ^{d-1}$-valued
martingale under 
$\bbP_{\Lambda_n}^{0 \to \ell_n}$. By Doob's $L^2$-inequality, we have 
\begin{equation} \label{doobs argument}
\bbE(\max_{j \leq n_0} |\sc X_j|^2) \leq 4 \bbE(|\sc X_{n_0}|^2) = 4 
\sum_{j=1}^{n_0} \bbE(( |\sc  X_{j} - \sc X_{j-1}|^2).
\end{equation}

The main technical part of the proof will consist in showing the following 
result: 

\begin{proposition} \label{moment bounds}
	Let $p \geq 1$. Then there exist $\alpha_0 < \infty$, $C < \infty$ and 
	$N \in \bbN$ so that for all $y \in \Lambda_n$, $n>N$, 
	$\alpha > \alpha_0$, $A \in \caA_y^s$, $k \in \bbR$ we have 
	\[
	\bbP_A^{y \to \ell_n} \Big( |X_1 - y| > k \log n \Big| \gamma 
	\in \caC_y \Big) \leq C k^{-p}.
	\]	
\end{proposition}
 
The proof is long, so we postpone it to the end of the section. Using 
Proposition \ref{moment bounds} with $p>2$, 
we find that for suitable $\alpha_c$, $C$ and $N$, 
\[
\begin{split}
& \bbE_{\Lambda_n}^{0 \to \ell_n} \Big(\frac{|X_{j} - X_{j-1}|^2}{(\log n)^{2}} 
\Big| X_j = y, R(X_j) = R \Big) \leq 
\int_1^\infty \bbP_R^{y \to \ell_n} (|X_1| \geq k \log n) k \, \dd k \\
& \leq 
C \int_1^\infty k^{1-p} \, \dd k = C / (p-2) =: C_p
\end{split}
\]
uniformly in $n>N$, $\alpha > \alpha_0$ and $R \in \caA_y^s$. 
By integrating over possible $y$ and $R$ we conclude 
$\bbE (| \sc X_j - \sc X_{j-1} |^2) \leq (\log n)^2 C_p$. 
Since
$\bar X_{i+1} \geq \bar X_i + \log n$ by construction, the martingale
takes at most
$n / \log n$ steps before hitting the point $\lim_{n \to \infty} \pi^n(0)$, 
after which it does no longer move. This, Chebyshevs inequality and 
\eqref{doobs argument} give 
\begin{equation}
\label{regen points constraint}
\bbP( \max_{i \in \bbN} |\sc X_i| > M \sqrt{n \log n} ) \leq 
\frac{1}{M^2 n \log n} \bbE( \max_{i \in \bbN} | \sc X_i|^2) \leq 
\frac{1}{M^2 n \log n} 4 \frac{n}{\log n} C_p (\log n)^2 = \frac{4 C_p}{M^2}.
\end{equation}
In conclusion, we now know that with high probability none of the special 
regeneration points $(X_i)$ is further than $\sqrt{n \log n}$ away from the 
line $\{\sc x = 0 \}$. 

In order to control the parts of $\gamma$ in between the $X_i$, 
we first apply Proposition \ref{moment bounds} with $k = n^{1/4}$ and 
$p=8$, and find 
\[
\bbP_{\Lambda_n}^{0 \to \ell_n} (| \fc  X_{i+1}- \fc X_i | > n^{1/4} 
\log n | X_i = y, R_i = A) \leq C_8 n^{-2}
\]
for all $y \in \Lambda_n$ and all $A \in \caA_y^s$. Since there are at most 
$n / \log n$ regeneration points, the union 
bound then gives (for $n$ large enough)
\[
\bbP_{\Lambda_n}^{0 \to \ell_n} ( \max_i | \fc X_{i+1} - \fc X_i | > 
n^{1/3} ) \leq C_8 n^{-1}.
\] 
We thus find 
\begin{equation} \label{almost there}
\begin{split}
& \bbP_{\Lambda_n}^{0 \to \ell_n} ( \max_{x \in \gamma} | \sc x | > 
2 M \sqrt{n \log n} ) \leq \frac{4 C_p}{M^2} + \frac{C_8}{n} + \\ 
& + 
\bbP_{\Lambda_n}^{0 \to \ell_n} \Big( \max_{x \in \gamma} | \sc x | > 
2 M \sqrt{n \log n}, \max_i | \fc X_{i+1} - \fc X_i | < 
n^{1/3}, \max_i |\sc X_i| > M \sqrt{n \log n} \Big).
\end{split}
\end{equation}

By definition, a regeneration point is the last time that $\gamma$ 
crosses a certain vertical hyperplane.
Proposition \ref{fundamental lemma} below deals with the probability that 
the length of $\gamma$ is large before it crosses such a hyperplane for 
the last time. 
Applying it with $\delta = 1$, $B = \Lambda_n$ and $L = n^{1/3}$, we see that 
the probability that the piece of $\gamma$ between $X_i$ and $X_{i+1}$ 
is longer than $2 n^{1/3}$ is less than $C \e{-c n^{1/3}}$ for each $i$. 
On the other hand, a path would need at least $2M \sqrt{n \log n}$ 
steps to exceed level $2 M \sqrt{n \log n}$ when starting (and ending) 
below level $M \sqrt{n \log n}$. Thus the second line in \eqref{almost there}
is bounded by $n \e{-c n^{1/3}}$. Taking $n$ so large that 
$n>M$ and $n \e{-c n^{1/3}} < \frac1M$, Theorem \ref{theo2} is proved,
provided we can prove the two propositions that we used. 

We start with stating and proving 
Proposition \ref{fundamental lemma}, which itself is a 
fundamental ingredient to the 
proof of Proposition \ref{moment bounds}. For its statement and also for 
later use, we introduce the following notation. For 
$L \in \bbN$ and $\pi \in \caS_A^{y \to \ell_n}$, set  
\[
x_L(\pi) := \max \{ x \in \gamma(\pi): \fc x = \fc y + L \}, 
\qquad \text{and} \qquad \rho_L(\pi) := \gamma(\pi) \setminus \Orb_\pi(x_L).
\]
$x_L$ is the last point at which the hyperplane $\{ z \in A: \fc z = \fc y 
+ L \}$ is crossed by $\gamma$, and $\rho_L$ is the piece of $\gamma$ that 
lies before that point. $|A|$ denotes the cardinality of a 
set $A$. Recall from above that $\caL_A^{x \to y}$ denotes the 
	set of all self-avoiding walks  that start in $x$, end in 
	$y$ and are contained in $A$. We write $\caL_A^{x \nearrow y}$ for the subset of those 
	$\gamma \in \caL_A^{x \to y}$ where $\gamma \cap \ell_{\fc x} = \{x\}$, 
	i.e. for the walks that never visit the hyperplane containing their 
	starting point again after their first step. We set 
	$\caL_A^{x \to \ell_n} := \bigcup_{y \in \ell_n} \caL_A^{x \to y}$, 
	and the same for $\caL_A^{x \nearrow \ell_n}$.

\begin{proposition} \label{fundamental lemma}
	For each $\delta > 0$, there exist $\alpha_0 < \infty$, $c>0$ and 
	$C < \infty$ such that for all $n > N$, $\alpha > \alpha_0$, 
	$L \in \bbN$ with $L > 2 \log n$, $y \in \Lambda_n$ with 
	$\bar y \leq n - L$, and $A, B \in \caA_y$ with $B \subset A$, we have 
	\[
	\bbP_A^{y \to \ell_n} \Big( | \rho_L | \geq (1 + \delta) L \Big| \rho_L 
	\subset B \Big) \leq C \e{- c L}.
	\]
\end{proposition}

\begin{proof}
	Let $A \in \caA_y$. 		
	We claim that there exist 
	$\alpha_0 < \infty$ and $N \in \bbN$ such that for all 
	$\alpha > \alpha_0$, $n>N$, $L > \log n$,   
	$h \in A$ with $\fc h = \fc y + L$, and   
	$\gamma \in \caL_{A}^{h \nearrow \ell_n}$,  we have  
	\begin{equation}
		\label{partition ratio}
		\frac12 \frac{Z(A)}{Z(\Lambda_n)} \leq \frac{Z(A \setminus \gamma)}
		{Z(\Lambda_n \setminus \gamma)} \leq 2 \frac{Z(A)}{Z(\Lambda_n)}.	 
	\end{equation}
	This follows from Corollary \ref{cor:decay3}: Since $A$ is admissible and 
	$\gamma$ stays to the right of $h$, we have 
	$d(\gamma, \Lambda_n \setminus A) > \log n$, and clearly 
	$|\gamma| |\Lambda_n \setminus A| \leq n^{2d}$. By Theorem \ref{theo1},
	we can choose $\alpha_0$ and $N \in \bbN$ 
	so large that for $\alpha > \alpha_0$ the 
	cycle length of spatial random permutations
	(in an arbitrary domain) is bounded by a random variable $\xi$ with 
	the property that the total population $W$ of a Galton-Watson 
	process with offspring distribution $\xi - 1$ satisfies 
	$\bbP(W>n) \leq \e{- 2 \kappa n}$ for all $n > N$, where 
	$\kappa = 2d + 1$. This implies 
	\[
	|\gamma| |\Lambda_n \setminus A| \e{- \kappa d(\gamma, 
	\Lambda_n \setminus A)} \leq \tfrac1n,
	\]
	and by further increasing $N$ if necessary Corollary \ref{cor:decay3} now 
	yields \eqref{partition ratio} for all $n > N$. 
	
	Now we fix $h \in \ell_{\fc y + L}$ and $\rho \in \caL_A^{y \to h}$ such 
	that $\rho \subset B$, 
	and define 
	\[
	\caS_{A \setminus \rho}^{h \nearrow \ell_n} := 
	\{ \pi \in \caS_{A \setminus \rho}^{h \to \ell_n}: \fc x > \fc h 
	\text{ for all } x \in \gamma(\pi) \setminus \{h\} \}.
	\] 
	Then, since $\rho \subset B$, we have 
	\begin{equation} \label{eq:rho_L}
	\begin{split}
	\bbP_A^{y \to \ell_n} (\rho_{L} & = \rho | \rho_L \in B) = 
	\e{-\alpha \| \rho \|} 
	\frac{1}{\sum_{\pi \in 
	\caS_{A}^{h \to \ell_n}} 
	\bbone \{\rho_L(\pi) \subset B \} 
	\e{- \alpha \caH_{A \setminus \gamma}(\pi)}} 
	\sum_{\pi \in 
	\caS_{A \setminus \gamma}^{h \nearrow \ell_n}} 
	\e{- \alpha \caH_{A \setminus \gamma}(\pi)}  \\
	& = \e{-\alpha \| \rho \| }
	\sum_{\gamma \in \caL_A^{h \nearrow \ell_n}} \e{- \alpha \| \gamma \|}
	Z(A \setminus (\gamma \cup \rho)) 
	\frac{1}{\sum_{\pi \in 
	\caS_{A}^{h \to \ell_n}} 
	\bbone \{\rho_L(\pi) \subset B \} 
	\e{- \alpha \caH_{A \setminus \gamma}(\pi)}}
	\end{split}
	\end{equation}
	For estimating the denominator, let 
	$\rho_0 = \{ x \in A: \fc y \leq \fc x \leq \fc y + L, \sc x = \sc y \}$
	be the straight line from $y$ to the hyperplane $\ell_h$, and let 
	$h_0 = \rho_0 \cap \ell_h$. Since $B$ is admissible, 
	$\rho_0 \subset B$, and thus  
	\[
	\begin{split}
	\sum_{\pi \in 
	\caS_{A}^{h \to \ell_n}} 
	\bbone \{\rho_L(\pi) \subset B \} 
	\e{- \alpha \caH_{A \setminus \gamma}(\pi)} 
	& \geq \sum_{\pi \in \caS_A^{y \to \ell_n}} 
	\bbone \{ \rho(\pi) = \rho_0 \} \e{ -\alpha \caH_A(\pi)} 
	= \e{-\alpha L} Z^{h_0 \to \ell_n}(A \setminus \rho_0) \\ 
	&  \geq 
	\e{-\alpha  L} Z^{h_0 \nearrow \ell_n}(A \setminus \rho_0)
	= \e{-\alpha  L} \sum_{\gamma \in \caL_A^{h_0 \to \ell_n}} \e{-\alpha \| \gamma \| } 
	Z(A \setminus (\rho_0 \cup \gamma)).
	\end{split}
	\]
	We now use Proposition \ref{prop:lowerboundpart} with $c_1(\alpha)$ as in 
	\eqref{eq:c1} as well as \eqref{partition ratio} and find 
	\[
	Z(A \setminus (\rho_0 \cup \gamma)) \geq c_1^{|\rho_0|} 
	Z(A \setminus \gamma) \geq \frac{c_1^L}{2} \frac{Z(A) 
	Z(\Lambda_n \setminus \gamma)}{Z(\Lambda_n)}.
	\]
	Going back into \eqref{eq:rho_L}, and using 
	\eqref{partition ratio} again, we also have that 
	\[
	Z(A \setminus (\gamma \cup \rho)) \leq Z(A \setminus \gamma) \leq 
	2 \frac{Z(A) Z(\Lambda_n \setminus \gamma)}{Z(\Lambda_n)}.
	\]
	Together, this gives 
	\[
	\bbP_A^{y \to \ell_n}(\rho_L = \rho | \rho_L \subset B) \leq 
	4 \e{-\alpha(\| \rho \| - L)} c_1^L \frac
	{ 
	\sum_{\gamma \in \caL_A^{h \nearrow \ell_n}} \e{-\alpha \| \gamma \|} 
	Z(\Lambda_n \setminus \gamma)
	}
	{ 
	\sum_{\gamma \in \caL_A^{h_0 \nearrow \ell_n}} \e{-\alpha \| \gamma \|} 
	Z(\Lambda_n \setminus \gamma)
	}
	.
	\]
	The final quotient is equal to one by translation invariance. Since 
	$c_1(\alpha)$ converges to $1$ as $\alpha \to \infty$, by making 
	$\alpha_0$ even bigger if necessary we achieve 
	\[
	\bbP_A^{y \to \ell_n}(\rho_L = \rho) \leq \e{-\alpha ( \| \rho \| - 
	(1 + \delta/2) L )}
	\]
	for all $\alpha > \alpha_0$. Now we sum over all allowed lengths of 
	$\rho$ and find 
	\[
	\bbP(|\rho_L| > (1 + \delta)L | \rho_L \subset B) \leq 
	\sum_{m = (1+\delta)L}^\infty (2d)^m \e{-\alpha ( m - 
	(1 + \delta/2) L )} = \e{- L ( \alpha \delta / 2 - ( 1 +\delta) \log (2 d) )} 
	\sum_{m=0}^\infty \e{ - m (\alpha - \log (2d) )}.
	\]
	(The crude bound $(2d)^{m}$ on the number of all self-avoiding walks of 
	length $m$ could of course be improved using the connective
	constant, but in the present situation we do not gain any insight 
	from that.) Now by further increasing $\alpha_0$ if necessary, 
	we see that the claim holds with $c = \alpha_0 \delta/2 -  ( 1 +\delta)  \log (2d)$, 
	and $C = 1 / (1 - \e{-\alpha_0 + \log (2d)})$. 
\end{proof}

The next statement, which is purely deterministic, shows that 
proposition \eqref{fundamental lemma} is useful for obtaining lower bounds
on the number of pre-regeneration points.  

\begin{lemma}
\label{lem:numbregpoint}
Let $L \in \bbN$, $A \subset \Lambda_n$, $x,y \in A$ such that 
$\fc y = \fc x + L$, and $\gamma \in \caL_A^{x \to y}$. 
Assume that $|\gamma| < ( 1 + \delta) \, L $.  
Then $\gamma$ has at least $ ( 1 - 3 \delta)\, L$ pre-regeneration points.
\end{lemma}
\begin{proof}
We recursively determine sequences of points in $\gamma$. We set $x_0 = x$.
For $i \geq 1$, we define 
\[
\begin{split}
	& y_i := \min \{ y \in \gamma: y \geq x_{i-1}, y 
	\text{ is not a pre-regeneration point} \}, \\
	& y_i' := \min \{ y \in \gamma: y > y_i, y \notin \caC_{y_i} \}, \\
	& m_i := \max \{ \fc x: x \in \gamma, x \leq y_i' \}, \\
	& x_{i+1} := \min \{ x \in \gamma: \fc x > m_i \}
\end{split}
\]
\begin{figure}
\centering
\includegraphics[scale=0.35]{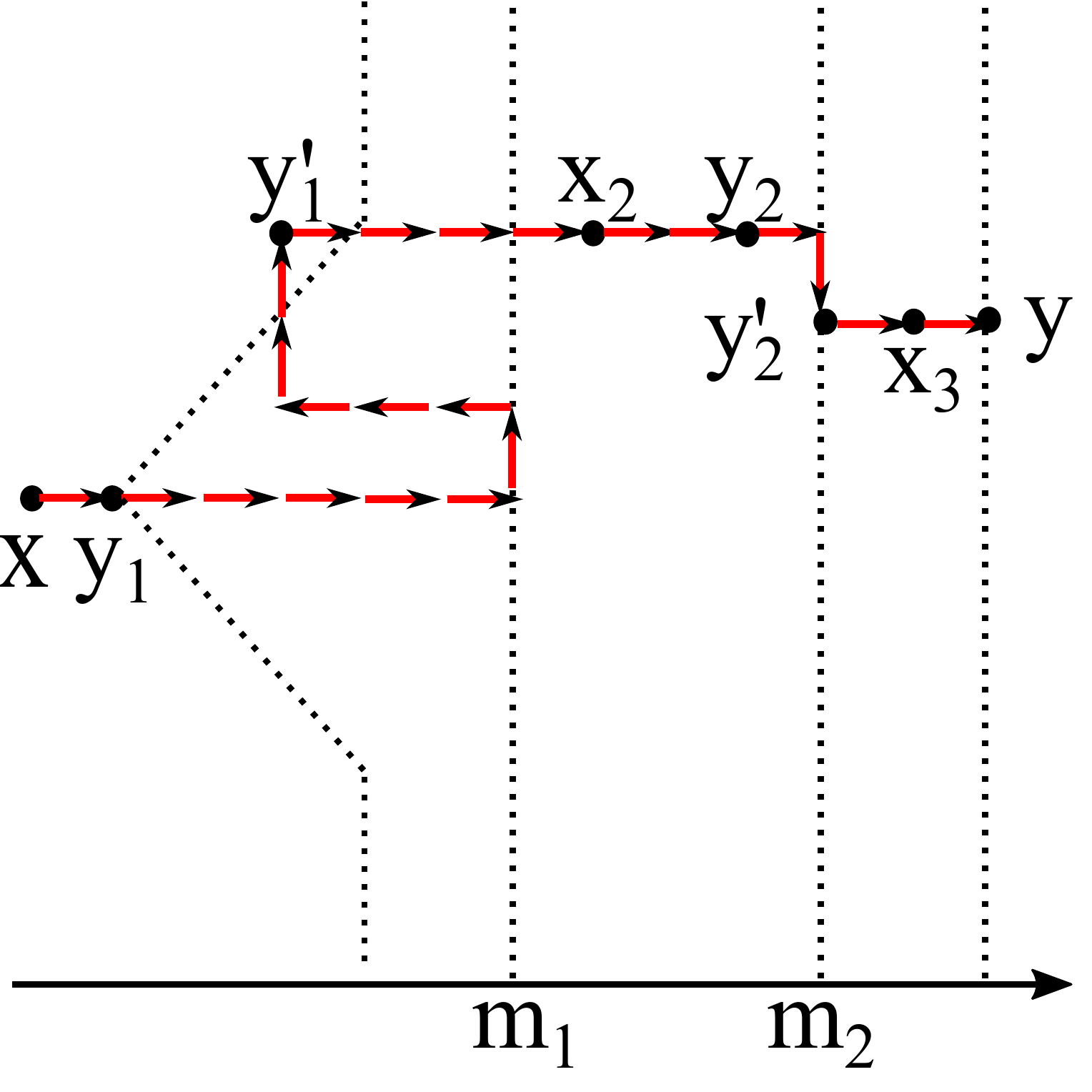}
\caption{A self-avoiding walk from $x$ to $y$.}
\label{Fig:pathandpoints}
\end{figure}
(see also Figure \ref{Fig:pathandpoints}).
Thus when following the path $\gamma$, we thus have the following situation: 
At each $x_i$, 
the path crosses the  hyperplane $\ell_{x_i}$ for the first time. 
Thereafter, we only see regeneration points until we hit $y_{i+1}$, although 
of course $y_{i+1} = x_i$ is possible. $y_{i+1}'$ is then the point where we 
actually notice that $y_{i+1}$ was not a pre-regeneration point (see also the 
following paragraph!). We then follow $\gamma$ further 
until we hit a vertical hyperplane that we have not seen before $y_{i+1}'$, 
and $x_{i+1}$ is the place where we cross that hyperplane. 

The crucial observation is that since 
$x_i$ is the first time a certain 
hyperplane is crossed, and since all points between $x_i$ and $y_i$ 
are pre-regeneration points, the reason that $y_i$ is not a pre-regeneration
point can only be that the path leaves the enhanced forward cone; the other 
possibility, namely that 
some previous part of the path has already crossed the hyperplane 
containing $y_i$, is ruled out. This means that indeed $y_i'$ is the place 
where we notice that $y_i$ was not a pre-regeneration point. It also means that 
between $y_i$ and $x_{i+1}$, the path stays strictly to the right of 
$\ell_{\fc y_i}$ and strictly to the left of $\ell_{\fc x_{i+1}}$. It is easy to 
see that in such a situation, at most a third of all the steps of the path 
can be 'to the right', i.e.\ between vertices $z, w$ with 
$\fc z = \fc w + 1$. Since $\gamma$ takes at least $L$ steps to the right 
there are at most $\delta L$ steps available for the other directions. Thus 
the union of all pieces of $\gamma$ that lie between some $y_i$ and the 
corresponding $x_{i+1}$ has less than $3 \delta L$ elements. Since all of the 
remaining at least $(1 - 3 \delta) L$ points are regeneration points by 
construction, the claim is proved.  
\end{proof}

The next step is to show that for sufficiently large $\alpha$, 
a pre-regeneration point is an actual regeneration point with uniformly 
positive probability. Recall that $\caR(\pi)$ denotes set of all 
regeneration points of $\pi$. 

\begin{proposition} 
\label{prop:uniformlyPositive}
	There exists $\alpha_0 < \infty$, $c > 0$ and $n_0 \in \bbN$ se that 
	for all $n \geq n_0$, $\alpha > \alpha_0$, $y \in \Lambda_n$, 
	$A \in \caA_y^w$, $\gamma_0 \in \caL_A^{y \to \ell_n}$, and 
	all pre-regeneration points $x \in \gamma_0$ 
	with $\fc x - \fc y \geq \log n$, we have 
	\[
	\bbP_A^{y \to \ell_n} (x \in \caR  | 
	\gamma = \gamma_0 ) > c .
	\]
\end{proposition}
\begin{proof}
	Since $A$ is weakly admissible and $\fc x - \fc y \geq \log n$, we have 
	$\{ z \in \Lambda_n: \sc z = \sc x \} \subset A$. Also, let us write 
	$\mathsf R(x,\pi)$ for the set of all subsets $R$ of $A$ that satisfy 
	properties (R1) to (R4) of regeneration sets. Since $x$ is a 
	pre-regeneration point of $\gamma_0$, we have $\Orb_\pi(x) \subset 
	\caC_x$ and $\gamma_0 \setminus \Orb_\pi(x) \subset 
	\{z: \fc z < \fc x \}$
	for all $\pi$ such that $\gamma(\pi) = \gamma_0$. Since on the other hand 
	$R \supset \caC_x$ for all $R \in \mathsf R(x,\pi)$, we have 
	\[
	\bbP_A^{y \to \ell_n} (x \in \caR | \gamma = \gamma_0 ) = 
	\bbP_{A \setminus \gamma_0} (\mathsf R(x, . ) \neq \emptyset ).
	\]	
	Let $\Phi$ be the group of reflections that leave the line $\{z: 
	\sc z = \sc x\}$ invariant, and for $\pi \in \caS_{A \setminus \gamma_0}$
	and $B \subset A$ recall that  $Q_B(\pi)$ is the minimal 
	$\pi$-invariant, 
	$\Phi$-compatible subset of $A$ that contains $B$. With  
	$\ell_x^- := \{ z \in A: \fc z < \fc x \}$, we then have 
	$\mathsf R(x,\pi) \neq \emptyset$ if and only if 
	$Q_{\ell_x^-}(\pi) \cap \caC_x = \emptyset$, and in this case 
	$R(x,\pi) = A \setminus Q_{\ell_x^-}(\pi)$. We can thus use Proposition 
	\ref{prop:size of invariant set}: we have $Q_{\ell_x^-}(\pi) \subset 
	\bigcup_{z: \fc z = \fc x - 1 } Q_{\{z\}}(\pi)$, and thus 
	\[
	\bbP_{A \setminus \gamma_0} (Q_{\ell_x^-} \cap \caC_x = \emptyset )
	\leq \sum_{z: \fc z = \fc x - 1} \bbP_{A \setminus \gamma_0} 
	(Q_{\{z\}} \cap \caC_x = \emptyset) 
	\leq \sum_{z: \fc z = \fc x - 1 } \bbP_{A \setminus \gamma_0} 
	(|Q_{\{z\}}| > \min \{ | \sc z  - \sc x |, \log n \} ). 	
	\]
	Pick $\tilde c < 1$ and let $\nu > d-1$ be so large that 
	\[
	\sum_{k=1}^\infty \e{- \nu k} | \{ z \in \Lambda_n: 
	\fc z = \fc x-1, k-1 \leq | \sc z - \sc x | < k \} | \leq \tilde c.
	\]
	Theorem \ref{theo1} and Propositions \ref{prop:size of invariant set} 
	and \ref{prop:exponential bound} guarantee that we can choose $\alpha_0$
	large enough so that for all $\alpha > \alpha_0$, all $z$ with 
	$\fc z = \fc x - 1$ and all $k \geq 1$, we  have that
	\[
	\bbP_{A \setminus \gamma_0} ( | Q_{\{z\}} | > k) \leq \e{- \nu k}.
	\]
	Then 
	\begin{equation} \label{no long things}
	\bbP_{A \setminus \gamma_0} (Q_{\ell_x^-} \cap \caC_x = \emptyset) 
	\leq \tilde c + \sum_{z: \fc z = \fc x-1} \bbP_{A \setminus \gamma_0}(|Q_{\{z\}}| > \log n) 
	\leq \tilde c + n^{d-1} n^{-\nu} \stackrel{n \to \infty}{\longrightarrow}
	\tilde c.
	\end{equation}
	This shows that for sufficiently large $n_0$ and all $n > n_0$, as well 
	as for all $\alpha > \alpha_0$, we have $\bbP_{A \setminus \gamma_0} 
	(\caR(x,.) \neq \emptyset) \geq (1 - \tilde c)/2 > 0$. 	
\end{proof}

We now have all the preparations in place to give the proof of Proposition
\ref{moment bounds} and thus conclude the proof of Theorem \ref{theo2}. 
Clearly, it suffices to prove the statement of Proposition 
\ref{moment bounds} only for large enough $k$.  
We fix $\delta < 1/4$, and for $k \geq 1$ and $L = k \log n$ we decompose
\begin{equation}
\label{final decomposition}
\begin{split}
\bbP_A^{y \to \ell_n} (|X_1 - y| > k \log n | \gamma \in \caC_y) 
\leq & \,\, \bbP_A^{y \to \ell_n} ( | \rho_L | \geq (1 + \delta) L | 
\rho_L \in \caC_y ) +\\ 
& + \bbP_A^{y \to \ell_n} ( |X_1 - y| > k \log n, |\rho_L| 
< (1 + \delta)L | \rho_L \in \caC_y).
\end{split}
\end{equation}
Here we used that $\rho_L \in \caC_y$ is equivalent to $\gamma \in \caC_y$ 
by the definitions of $\caC_y$ and $\rho_L$ and the fact that 
$L \geq \log n$. Since $\caC_y \in \caA_y$, we can apply Proposition 
\ref{fundamental lemma} and find 
\[
\bbP_A^{y \to \ell_n} ( | \rho_L | \geq (1 + \delta) L | \rho_L 
\in \caC_y) \leq C \e{- c k \log n} \leq C \e{- c k}.
\]
uniformly in $k$ for all large enough $\alpha$ and $n$. 

It remains to estimate the second term in \eqref{final decomposition}. 
For this, we use the trivial equality 
\[
\begin{split}
& \bbP_A^{y \to \ell_n} ( |X_1 - y| > k \log n, |\rho_L| 
< (1 + \delta)L | \rho_L \in \caC_y) \\ 
& = \frac{
\sum_{\rho \subset \caC_y: |\rho| < (1 + \delta) L} 
\bbP_A^{y \to \ell_n} (|X_1 - y| > K \log n | \rho_L = \rho) 
\bbP_A^{y \rightarrow \ell_n}(\rho_L = \rho)
}{
\sum_{\rho \subset \caC_y} \bbP_A^{y \rightarrow \ell_n}(\rho_L = \rho)
}
\end{split}
\]
in order to see that the claim will be shown once we prove that there exists
$\tilde C < \infty$ such that uniformly in $\eps, n$ and $\rho \subset A$ 
such that $|\rho| < (1 + \delta) L$ and such that  
$\rho$ is the trace of some self-avoiding walk from $y$ to some element of 
$\ell_{\fc y + L}$, we have 
\begin{equation}
\label{last thing to prove}
	\bbP_A^{y \to \ell_n} (|X_1 - y| > K \log n | \rho_L = \rho) 
	\leq \tilde C k^{-p}	.
\end{equation}
Lemma \ref{lem:numbregpoint} tells us that such a $\rho$ must have at least 
$(1 - 3 \delta) k \log n \geq \frac14 k \log n$ pre-regeneration points. 
For $k>8$, at least $(k/4 - 1) \log n \geq \frac{k}{8} \log n$ of them 
satisfy $\fc x - \fc y \geq 
\log n$. If any of these is a regeneration point, then $X_1$ is the minimal 
such point, and thus $|X_1 - y| \leq k \log n$. 

Let therefore $\rho$ be the trace of some self-avoiding walk from $y$ to 
an element of $\ell_{\fc y + L}$, assume that it has $M > \log n$ 
pre-regeneration points that have horizontal distance larger than 
$\log n$ from $y$, and denote these points by $(x_i)_{i \leq M}$. 
By Proposition \ref{prop:uniformlyPositive}, each $x_i$ is a regeneration 
point with uniformly positive probability. If the events 
$\{\pi: x_i \in \caR(\pi)\}$ would be independent, this would 
already conclude the proof, as we have many of these points. 
Unfortunately, they are not independent, and so we will have to work harder. 
(Although we can not prove it, it is actually reasonable to conjecture that 
the events $\{ x_i \notin \caR \}$ are  
positively correlated since the fact that $x_i$ is a pre-regeneration point
but not a regeneration point means that there is some long cycle of $\pi$
that prevents the existence of a regeneration surface; such a long cycle 
may still impact the next pre-regeneration point as well.)

The solution is to separate pre-regeneration points by invariant sets. 
Let $r_1 = x_1$, and
\[
r_j = \min \{ x_j: \fc x_j > \fc r_{j-1} + 2 \log n \}.
\]
Since we started with at least $\frac18 k \log n$ pre-regeneration points
$(x_i)$, and since at most $2 \log n$ of them are between two consecutive 
$r_i$, we still retain at least $\tilde k \geq k/16$ points $(r_i)_{i \leq 
\tilde k}$. For $\pi \in \caS_{A \setminus \gamma}$, denote again 
by $\mathsf R(x,\pi)$ the set of all subsets of $A$ satisfying properties 
(R1) to (R4). With  
\[
\caN(x) := \{ \pi \in \caS_{A \setminus \gamma}: \mathsf R(x,\pi) = 
	\emptyset \}, \quad \caN_i := \caN(r_i), 
\]
for any $\gamma \subset A$ so that 
$\rho_L(\gamma) = \rho$, we then have 
\[
\bbP_A^{y \to \ell_n} (|X_1 - y| > K \log n | \rho_L = \rho) = 
\bbP_{A \setminus \gamma} \Big( \bigcap_{j\geq 1} \caN(x_j) \Big) 
\leq \bbP_{A \setminus \gamma} \Big( \bigcap_{i = 1}^{\tilde k} 
\caN_i \Big).
\]
We further define 
$w_i := \fc r_i + \log n$, and 
\[
\ell_{w_i}^- := \{z \in A: \fc z < \fc w_i \}, \qquad 
Q_i := \Orb_\pi (\ell_{w_i}^- \cap (A
\setminus \gamma)), \qquad 
W_i := \{z \in A: \fc z \geq \fc w_{i} + \log n \}.
\]
Let $m \leq \tilde k$. We decompose, 
\[
\bbP_{A \setminus \gamma} \Big( \bigcap_{i \geq 1}^m \caN_i \Big) 
\leq \bbP_{A \setminus \gamma}(Q_{m-1} \cap W_{m-1} \neq \emptyset) 
+ \bbP_{A \setminus \gamma}\Big( Q_{m-1} \cap W_{m-1} = \emptyset,
\bigcap_{i \geq 1}^m \caN_i \Big). 
\]
By the same reasoning that led to the bound on the second term in 
\eqref{no long things}, the first term 
above is bounded by $n^{d-1-\nu}$, where we can choose $\nu$ as large as 
needed by making $\alpha$ large. For the second term we use the strong 
Markov property of SRP. Let $f(\pi) = \bbone\{\pi \in \caN_m\}$, and note 
that $\bigcap_{i=1}^{m-1} \caN_i \in \caF_{Q_{m-1}}$ by the definition of the
$Q_i$ and the strict admissibility of regeneration sets. Thus 
\[
\bbP_{A \setminus \gamma} \Big(Q_{m-1} \cap W_{m-1} = \emptyset, \bigcap_{i 
\geq 1}^m \caN_i \Big) = \bbE_{A \setminus \gamma} \Big( 
\bbE_{A \setminus \gamma} \big( f   \big| \caF_{Q_{m-1}} \big) 
\bbone \{ Q_{m-1} \cap W_{m-1} = \emptyset 
\} \bbone \big\{ \bigcap_{i=1}^{m-1} \caN_i \big\} 
\Big),
\]
and Proposition \ref{prop:strong Markov} gives 
\[
\bbE_{A \setminus \gamma} \big( f   \big| \caF_{Q_{m-1}} \big) 
\bbone \{ Q_{m-1} \cap W_{m-1} = \emptyset 
\} = \bbP_{Q_{m-1}}(\caN_m)  \bbone \{ Q_{m-1} \cap W_{m-1} = \emptyset 
\}
\]
almost surely. By Proposition \ref{prop:uniformlyPositive} we find 
$\bbP_{Q_{m-1}(\pi)}(\caN_m) \leq 1-c$ uniformly in all allowed 
$Q_{m-1}(\pi)$ (note that by definition, $Q_{m-1}(\pi)$ is weakly 
admissible for all $\pi$ such that 
$Q_{m-1}(\pi)  \cap W_{m-1} = \emptyset$), and thus we conclude 
\[
\bbP_{A \setminus \gamma} \Big( \bigcap_{i=1}^m \caN_i \Big) 
\leq (1-c) \bbP_{A \setminus \gamma} \Big( 
\bigcap_{i=1}^{m-1} \caN_i \Big) + n^{d-1-\eta},
\]
for all $m \leq \tilde k$. Therefore, by induction, 
\[
\bbP_{A \setminus \gamma} \Big( \bigcap_{i=1}^{\tilde k} \caN_i \Big) 
\leq \sum_{j=0}^\infty (1-c)^j n^{d-1-\eta} = \frac{n^{d-1-\eta}}{c}.
\]
It remains to choose $\alpha_0$ so large that for $\alpha > \alpha_0$, 
we have $d - 1 - \eta < -p$. Then for sufficiently large $n_0$ and all 
$n > n_0$, we have 
\[
\bbP_A^{y \to \ell_n} ( |X_1 - y| > k \log n | \gamma \in \caC_y) 
\leq C \e{-c k} + \frac1c n^{-p}.
\]
For $k \leq n$, this proves the claim of Proposition \ref{moment bounds}; 
the last step in the proof consist in observing that for $k>n$, the 
required probability is trivially equal to zero when $k > d$. 
Thus Proposition \ref{moment bounds}, and thus Theorem \ref{theo2}, is 
proved. 

Let us finish this paper by looking back over the proof and 
identifying the place where our estimates are not good enough 
to provide diffusive scaling. The most basic place where this 
happens is Proposition \ref{fundamental lemma}: here already, 
we can only prove a bound on the length of $\rho_L$ (and thus 
many regeneration points) when $L$ is of the order of $\log n$. 
This in turn is due to the inequality \eqref{partition ratio}, 
which needs the separation of the part of $\gamma$ to the right of 
$\ell_{\fc h}$ from the set $\Lambda_n \setminus A$ by at least 
a distance of $\log n$. This is because \eqref{partition ratio} relies 
on Proposition \ref{prop:decay of boundary 2}, and in this result 
the sum that leads to the constant $D$ needs to be controlled by
sufficinetly large distances between $A$ and $V_0 \setminus V_1$, 
especially when the number of points in either of them diverges. 
The rood cause of the problem is that while we do have exponential 
decay of correlations, the decay is not uniform in the size of the 
sets between which the correlations are measured. See also the 
discussion in the paragraph before Proposition 
\ref{prop:decay of boundary 2}. The latter proposition 
already goes some way towards solving this 
problem, but as it turns out it is not quite enough for obtaining 
optimal results.

\section*{Appendix}
\label{sect:Appendix}
Given two finite or countable infinite sequences of 
non-negative integer-valued random variables,
$\left(  M^i_n \right)_{n \in I \subset \mathbb{N}}$, $i= 1,2$,
we say that 
$\left(  M^2_n \right)_{n \in I }$
is stochastically larger than
$\left(  M^1_n \right)_{n \in I }$
and we write
$$
\left(  M^1_n \right)_{n \in I } \overset{d}{\preceq}  \left(  M^1_n \right)_{n \in I },
$$
if for any finite sequence of non-negative integers $n_0$, $n_1$, $\ldots$, $n_k$,
we have that,
$$
P_1(M^1_0 > n_0, \, M^1_1 > n_1,\,  \ldots M^1_k > n_k) \leq P_2(M^2_0 > n_0,\, M^2_1 > n_1,\,  \ldots M^2_k > n_k),
$$
where $P_i( \, \cdot \ )$ is the law of the $i$-th sequence, $i=1,2$.
We are now ready to state our Comparison Lemma.
\begin{lemma}[\textbf{Comparison Lemma}]
\label{lem:comparisonlemma}
Let $\left(  M^i_n \right)_{n \in \mathbb{N} }$, $i \in \{1,2\}$,  be two non-negative,
integer valued stochastic processes, defined in two different probability spaces,
$( \Sigma_1, \mathcal{F}^1, P_1)$ and 
$(\Sigma_2, \mathcal{F}^2, P_2)$.
Let $\left(\mathcal{F}_n^i\right)_{n \in \mathbb{N}}$, $i \in \{1,2\}$ be their filtrations. Assume
that $\left(  M^2_n \right)_{n \in \mathbb{N} }$ is a Markov chain 
and let $\left(  M^{2, \lambda}_n \right)_{n \in \mathbb{N} }$ be used
to denote the Markov chain with initial state $M^2_0$ distributed according
to  the probability measure on the non-negative integers $\lambda( \, \cdot \,)$. 
Assume that 
\begin{itemize}
\item (\textbf{\textit{Ass. 1}}) The Markov chain is such that,
$$ \lambda  \overset{d}{\preceq} \lambda^{\prime} \implies \left(  M^{2, \lambda}_n \right)_{n \in \mathbb{N} } \overset{d}{\preceq}  \left(  M^{2, \lambda^{\prime}}_n \right)_{n \in \mathbb{N} } $$
\item   (\textbf{\textit{Ass. 2}})  For all integers $\ell,  m, n 	\geq 0$,  
and $\omega \in \Sigma^1$ such that $M^1_{n}(\omega)=m$,
$$  P_1( M^1_{n+1} > \ell \bigm| \mathcal{F}_{n}^1)(\omega) \leq  P_2( M^{2}_{n+1} > \ell \Large |  M^{2}_{n}=m).$$
\item (\textbf{\textit{Ass. 3}}) For a given initial distribution $\lambda^{*} ( \, \, \cdot \, \, )$,
$$
 M^1_0   \overset{d}{\preceq}   M^{2, \lambda^{*}}_0.
$$
\end{itemize}
Then, we have that 
\begin{itemize}
\item (\textbf{\textit{Concl. a}})  
The Markov chain $(M^{2, \lambda^*}_n)_{n \in \mathbb{N}}$ is stochastically larger than $(M^1_n)_{n \in \mathbb{N}}$,
$$ (M^1_n)_{n \in \mathbb{N}}   \overset{d}{\preceq} (M^{2, \lambda^{*}}_n)_{n \in \mathbb{N}},$$
\item (\textbf{\textit{Concl. b}}) For all $k \in \mathbb{N}$, 
$$  \sum\limits_{j=0}^{k} M^1_j  \overset{d}{\preceq}  \sum\limits_{j=0}^{k} M^{2, \lambda^{*}}_j.$$ 
\end{itemize}
\end{lemma}
\begin{proof}
Let us denote by $\overline{P}( \, \cdot \, ) = P_1( \, \cdot \, ) \times P_2( \, \cdot \, )$ and by $\overline{E}( \, \cdot \, )$ the expectation
of $\overline{P}( \, \cdot \ )$.
We start by proving the statement (a) of the lemma. Fix a finite sequence of non-negative integers,
$n_1$, $n_2$, $\ldots$, $n_k$. We replace steps of the process $1$ by steps of the 
process $2$ one by one, showing that the probability of the event $\mathbbm{1} \{ \,  M^1_k > n_k, M^1_{k-1} > n_{k-1}, \ldots M^1_0 > n_0 \, \}$
cannot decrease at each replacement. Here is our first replacement,
\begin{multline}
P_1(M^1_k > n_k, M^1_{k-1} > n_{k-1}, \ldots M^1_0 > n_0)   =  \\
\overline{E} \, \Big( \mathbbm{1} \{ M^1_k > n_k, M^1_{k-1} > n_{k-1}, \ldots M^1_0 > n_0 \} \Big) =  \\ 
\overline{E} \, \Big( \overline{E} (\mathbbm{1} \{ M^1_k > n_k\} \, \big  | \, \mathcal{F}^1_{k-1}) \, \,  \mathbbm{1}\{  M^1_{k-1} > n_{k-1}, \ldots M^1_0 > n_0 \} \Big)    \leq  \\
\overline{E} \, \Big( \overline{E} (\mathbbm{1} \{ M^2_1 > n_k\}  \, \big |  \, M^2_{0} = M^1_{k-1}) \, \,  \mathbbm{1}\{ M^1_{k-1} > n_{k-1}, \ldots M^1_0 > n_0 \} \Big)  = \\
\overline{E} \, \Big( \mathbbm{1} \{ M^{2, \delta_{M^1_{k-1}}}_1 > n_k, M^1_{k-1} > n_{k-1}, \ldots M^1_0 > n_0 \} \Big),
\end{multline}
where for the second equality we used the definition of conditional expectation and for the first inequality we used our second assumption in the statement of the lemma. In the last equality we denote by $M^{2, \delta_{M^1_{k-1}}}_1$ the first step of the Markov chain which
starts at time $0$ from the value $M^1_{k-1}$ (then, the outer expectation integrates over the $M^1_{k-1}$).
Here is our second replacement,
\begin{multline}
\overline{E} \, \Big( \mathbbm{1} \{ M^{2, \delta_{M^1_{k-1}}}_1 > n_k, M^1_{k-1} > n_{k-1}, \ldots M^1_0 > n_0 \} \Big) = \\
\overline{E} \, \Big( \overline{E} \big(\mathbbm{1} \{ M^{2, \delta_{M^1_{k-1}}}_1 > n_k, M^1_{k-1}> n_{k-1} \} \, \big | \, \mathcal{F}^1_{k-2} \big) \, \,  \mathbbm{1}\{ M^1_{k-2} > n_{k-2}, \ldots M^1_0 > n_0\} \Big)  \leq  \\
\overline{E} \, \Big( {E}^2 \big ( \mathbbm{1} \{ M^{2}_2 > n_k, M^2_{1}> n_{k-1} \} \, \big | \,  M^2_{0} = M^1_{k-2} \big) \, \,  \mathbbm{1}\{ M^1_{k-2} > n_{k-2}, \ldots M^1_0 > n_0 \} \Big) = \\
\overline{E}  \Big(\mathbbm{1} \{ M^{2, \delta_{M^1_{k-2}}}_2 > n_k, M^{2, \delta_{M^1_{k-2}}}_{1} > n_{k-1}, 
M^1_{k-2} > n_{k-2}, \ldots M^1_0 > n_0 \} \Big).
\end{multline}
In the expression on the left-hand side of the first inequality, 
$M^{2, \delta_{M^1_{k-1}}}_1$  can be interpreted as the step $n=1$ of the  Markov chain which starts from the initial state 
distributed according to $M^1_{k-1}$ conditioned on $\mathcal{F}^1_{k-2}$
(then, the outer expectation integrates over the $\mathcal{F}^1_{k-2}$).
The inequality holds  as we replace such an initial state for the Markov chain
by a new state that, by our second assumption, is stochastically larger, namely by $M^2_1$ conditioned on $M^2_0 = M^1_{k-2}$.
Then, our first assumption in the statement of the lemma
guarantees that the whole Markov chain (not only the step $1$, but also
the following steps) is stochastically larger.
After $k-1$ iterations we get the first inequality in the expression below.
\begin{multline}
P_1 \Big( M^1_k > n_k, M^1_{k-1} > n_{k-1}, \ldots M^1_0 > n_0 \Big) \leq   \\
\overline{E}  \Big(\mathbbm{1} \{ M^{2, \delta_{M^1_{0}}}_{k} > n_k, M^{2, \delta_{M^1_{0}}}_{k-1} > n_{k-1}, 
\ldots, M^{2, \delta_{M^1_{0}}}_{1} > n_{2},  M^{1}_0> n_0 \} \Big)
\leq   \\
\overline{E}  \left(\mathbbm{1} \{ M^{2, \lambda^{*}}_{k} > n_k, M^{2, \lambda^{*}}_{k-1} > n_{k-1}, 
\ldots, M^{2, \lambda^{*}}_{1} > n_{2},  M^{2, \lambda^{*}}_0> n_0 \} \right),
\end{multline}
For the last inequality we used our Assumptions 1 and 3.
Indeed, on the left-hand side of the second inequality we have the Markov chain
$(M^2_n)_{n \in \mathbb{N}}$
starting from initial state $M^{1}_0$,
on the right-hand side we have the same Markov chain
starting from initial state $M^{2, \lambda^{*}}_0$,
which is is stochastically larger than $M^{1}_0$
by our Assumption 3.

The proof of the part (b) of the lemma goes along the same replacement
procedure as in the part (a).
Let use define $W^{i}_k = \sum\limits_{n=0}^{k} M^{i}_k$, $i \in \{1,2\}$.
We just illustrate our  first and second replacements below.
\begin{equation}
\begin{split}
P_1 \left( W^1_k > n \right) & = E^1 \Big( \mathbbm{1}\{ W^1_{k-1} + M_k^1 > n \} \Big) \\ & =  
\overline{E} \, \Big(    \, \, \, E^1 \big(  \, \,  \mathbbm{1}\{ M^1_{k} > n - W^1_{k-1} \}  \, \,  \big | \, \,  \mathcal{F}^1_{k-1}, M^1_{k-1} \, \, \, \big) \, \, \Big)  \\ & \leq  \overline{E}\Big(    \, \,  \overline{E}\,  \big( \, \,  \mathbbm{1} \{ M^2_{1} > n - W^1_{k-1} \}  \, \, \, \big | \, \, \,    M^2_{0}  = M^1_{k-1}, \mathcal{F}^1_{k-1}\, \, \big)  \, \, \Big) \\ & = 
\overline{E} \Big( \, \,  \mathbbm{1} \{ M^{2, \delta_{M^1_{k-1}}}_1 + W^1_{k-1} + > n\}  \, \, \Big),
\end{split}
\end{equation}
where we used our Assumption 2 for the inequality.
Here is the second step, where we use our Assumptions 1 and 2.
\begin{equation}
\begin{split}
\overline{E} \Big( \, \,  \mathbbm{1} \{  M^{2, \delta_{M^1_{k-1}}}_1 + M^1_{k-1} +  W^1_{k-2} > n\}  \, \, \Big)  & = 
\overline{E} \Big( \, \, \overline{E} \big(  \, \mathbbm{1} \{  M^{2, \delta_{M^1_{k-1}}}_1 + M^1_{k-1}  > n - W^1_{k-2} \}  \,  \, \big |   \mathcal{F}^1_{k-2}, M_{k-2}^1\, \big) \,  \Big)   \\
& =\overline{E} \Big( \, \, \overline{E} \big(  \, \mathbbm{1} \{  M^{2, \delta_{M^1_{k-2}}}_2 + M^{2, \delta_{M^1_{k-2}}}_{1}  > n - W^1_{k-2}\}  \, \, \big |   \mathcal{F}^1_{k-2}, M_{k-2}^1  \big) \, \,  \Big)   \\ 
& \leq \overline{E} \Big( \, \,  \mathbbm{1} \{ M^{2, \delta_{M^1_{k-2}}}_2 + M^{2, \delta_{M^1_{k-2}}}_1 +  W^1_{k-2} > n\}  \, \, \Big).
\end{split}
\end{equation}
The proof of the statement (b) of our Comparison Lemma follows by iteration.
For the last step, we use our Assumption 3, the same as for the proof of our statement (a).
\end{proof}
The last result that we present in this section 
is that the simple Galton-Watson process is a Markov chain
satisfying the Assumption 1 on the process $(M^2_j)_{j \in \mathbb{N}}$ in Lemma \ref{lem:comparisonlemma}.
\begin{lemma}
\label{lem:stochdomgalton}
 If $\Xi$ and $\Xi^{\prime}$
are two independent random variables such that 
$ \Xi \overset{d}{\preceq} \Xi^{\prime}$,
then
$(Z^{\Xi}_j)_{j \in \mathbb{N}}  \overset{d}{\preceq}  (Z^{\Xi^{\prime}}_j)_{j \in \mathbb{N}} $,
where $(Z^{\Xi}_j)_{j \in \mathbb{N}}$ is a simple Galton-Watson with offspring
distribution $\xi$.
\end{lemma}
We are now ready to prove the previous lemma.
\begin{proof}[Proof of Lemma \ref{lem:stochdomgalton}]
We couple the two Galton-Watson process in the same probability space.
In this probability space we define an infinite array of independent random 
variables $\boldsymbol{\xi} = \left(\xi_{j,k}\right)_{j \geq 1, \, k\geq 1}$
having the same distribution 
as $\xi$.
For any $m \in \mathbb{N}$
and $\boldsymbol{\xi}$,
we define the deterministic sequence $\left(  Z_j(m, \boldsymbol{\xi}) \right)_{j \in \mathbb{N}}$,
which represents a Galton-Watson process starting from initial 
condition $m$, where the $k$-th individual
of the $j$-th generation generates an offspring of size $\xi_{j,k}$.
Let us define the sequence precisely.
By definition, we set $Z_0 \left( m, \boldsymbol{\xi} \right) = m$.
We label individuals 
of the first generation according to an arbitrary order
using the integers $1$, $2$, $\ldots$, $m$.
After having labelled individuals of the first generation,
each individual $k$
generates an offspring of size $\xi_{1, k}$.
We set $Z_1 \left( m, \boldsymbol{\xi} \right)$ as the 
number of individuals of the generation $j=1$.
We label individuals of the second generation by using  integer numbers increasing
one by one, $1$, $2$, $3$, $\ldots$, in such a way that, for any two individuals $A$ and $B$,
if the label of the parent of $A$ is smaller than the label of the parent of $B$,
then the label of $A$ is smaller than the label of $B$.
Once  labels are assigned, each individual $k$ of the second generation generates
an offspring of size $\xi_{2,k}$
and variables $\xi_{2, Z_1+1}$, $\xi_{2, Z_1 + 2}$, $\ldots$ remain unused.
We set $Z_2 \left( m, \boldsymbol{\xi} \right)$ as the 
number of individuals of the generation $j=2$.

At any step $j$, we assign labels according to such a rule
and we use the variables $\xi_{j,k}$  for the children of each individual $k$.
If for some $j$,  $Z_j(m, \boldsymbol{\xi}) = 0$,
then by definition $Z_{j+1}(m, \boldsymbol{\xi}) = Z_{j+2}(m, \boldsymbol{\xi}) = \ldots =  0$.
The process can be seen as a tree where each individual
is connected to its children by a directed edge.
By construction, we have that for 
any realization of the array $\boldsymbol{\xi}$,
for all $j \in \mathbb{N}$,
\begin{equation}
\label{eq:monotonicitygalton1}
 Z_j\left( m, \boldsymbol{\xi}) \right)_{j \in \mathbb{N}} \leq  Z_j \left( m^{\prime}, \boldsymbol{\xi}) \right)_{j \in \mathbb{N}}
\end{equation}
whenever $m^{\prime} \geq m$,
as the tree corresponding to the Galton-Watson process with initial condition
$m$ is a sub-graph of the tree corresponding to the Galton-Watson 
process with initial condition $m^{\prime}$.

Let us now fix a $k \in \mathbb{N}$ and a sequence of non-negative integers $n_0$, $n_1$, $\ldots$, $n_k$.
For any fixed $\boldsymbol{\xi}$, we let $M(\boldsymbol{\xi}, n_0, \ldots n_k)$ be the smallest
$m$ such that 
\begin{equation}
\label{eq:condition}
 Z_j(m, \boldsymbol{\xi})  > n_j \mbox{ for all $j$ between $0$ and $k$}. 
 \end{equation}
From the monotonicity property (\ref{eq:monotonicitygalton1}), we have that if (\ref{eq:condition})
holds for $m$, then it holds for all $m^{\prime} \geq m$. 
Let then $\tilde{P}( \, \, \cdot \, \, )$ be the law of the 
array $\boldsymbol{\xi}$, of $\Xi$ and $\Xi^{\prime}$ together.
We have that,
\begin{multline}
\tilde{P}\left(\Xi > n_0, \, Z_1(\Xi, \boldsymbol{\xi}) > n_1, \, \ldots, \, Z_k(\Xi, \boldsymbol{\xi}) > n_{k} \right)  = 
\tilde{P} \left(\Xi \geq M(\boldsymbol{\xi}, n_0, \ldots n_k) \right) = \\
\sum\limits_{\boldsymbol{\xi}} \tilde{P} \left( \Xi \geq M(\boldsymbol{\xi}, n_0, \ldots n_k) \, \,  \big | \, \, \boldsymbol{\xi} \right)
\tilde{P} \left( \, \, \boldsymbol{\xi} \, \, \right)
\leq  
\sum\limits_{\boldsymbol{\xi}} \tilde{P} \left( \Xi^{\prime} \geq M(\boldsymbol{\xi}, n_0, \ldots n_k) \, \,  \big | \, \, \boldsymbol{\xi} \right)
\tilde{P} \left( \, \, \boldsymbol{\xi} \, \, \right) =  \\
\tilde{P} \left( \Xi^{\prime} \geq M(\boldsymbol{\xi}, n_0, \ldots n_k) \right)
= 
\tilde{P}\left(\Xi^{\prime} > n_0, \, Z_1(\Xi^{\prime}, \boldsymbol{\xi}) > n_1, \, \ldots, \, Z_k(\Xi^{\prime}, \boldsymbol{\xi}) > n_{k} \right),
\end{multline}
For the first inequality we used that $\Xi^{\prime}$ is stochastically larger than $\Xi$ by definition.
As the two processes that we constructed in the space of $\tilde{P}( \, \, \cdot \, \, )$ 
have the same distribution as the Galton-Watson processes in the statement of the lemma,
the previous inequality concludes the proof of the claim.
\end{proof}

\section*{Acknowledgements}
This work was supported by the DFG - Deutsche Forschungsgemeinschaft 
(grant number:  BE 5267/1).
The authors are grateful to  Ron Peled for his very interesting comments.

\end{document}